\documentclass[12pt]{amsart}

\usepackage{amssymb, a4, mathrsfs}

\usepackage{amssymb,epic,mathrsfs}
\usepackage{pstricks}
\usepackage{pstricks,pst-node,pst-plot}
\newpsobject{showgrid}{psgrid}{subgriddiv=1,griddots=5,gridlabels=6pt}
\usepackage{color}
\newtheorem{thm}{Theorem}[section]

\newtheorem{prop}[thm]{Proposition}
\newtheorem{lem}[thm]{Lemma}
\newtheorem{cor}[thm]{Corollary}

\theoremstyle{definition}
\newtheorem{defn}[thm]{Definition}
\newtheorem{rem}[thm]{Remark}
\newtheorem{ex}[thm]{Example}

\usepackage{pstricks,pst-node,pst-plot}
\theoremstyle{definition}

\numberwithin{thm}{section}
\numberwithin{equation}{section}

\newcommand{\el}{\mbox{$\mathcal{L}$}}
\newcommand{\dee}{\mbox{$\mathcal{D}$}}

\newcommand{\ar}{\mbox{$\mathcal R$}}

\newcommand{\leqel}{\mbox{$\leq _{\mathcal L }$}}
\newcommand{\leqar}{\mbox{$\leq _{\mathcal R}$}}

\newcommand{\leqarq}{\mbox{$\leq _{\mathcal{R}}^Q$}}
\newcommand{\leqelqs}{\mbox{$\leq _{\mathcal{L},S}^Q$}}
\newcommand{\leqarqs}{\mbox{$\leq _{\mathcal{R},S}^Q$}}

\newcommand{\els}{\mbox{${\mathcal L}^{\ast}$}}

\newcommand{\ars}{\mbox{${\mathcal R}^{\ast}$}}

\newcommand{\dom}{\textup{dom}\,}

\newcommand{\g}{\gamma}
\renewcommand{\a}{\alpha}
\renewcommand{\b}{\beta}

\newcommand{\im}{\operatorname{im}}

\renewcommand{\dom}{\operatorname{dom}}

\begin{document}

\title{Inverse semigroups of left I-quotients}
\author{Nassraddin Ghroda}
\author{Victoria Gould}
\address{Department of Mathematics\\ University of York\\ Heslington
\\ York\\ YO10 5DD\\ UK}
\email{varg1@york.ac.uk}
\email{ng521@york.ac.uk}

\date{\today}
\keywords{ample semigroup, I-order, I-quotients,
inverse hull}
\begin{abstract} We examine, in a general setting,  a notion of inverse
semigroup of left
quotients, which we call {\em left I-quotients}. This concept has appeared, and has been used, as far back as Clifford's seminal work describing bisimple inverse
monoids in terms of their right unit subsemigroups. As a consequence of our approach, we find a straightforward way of extending Clifford's work to bisimple inverse semigroups (a step that has previously proved to be awkward).  We also put some earlier work on Gantos into a wider and clearer context, and  pave the way for further progress.
\end{abstract}
\maketitle

\section*{Introduction}\label{sec:intro}
The notion of quotient plays an important role in algebra.
As far as semigroup theory is concerned, it occurs in its
 simplest form as the concept of a {\em group} $G$ of left 
quotients of a subsemigroup $S$, which requires that
every $g\in G$ can be written as $g=a^{-1}b$ where
$a,b\in S$. A well known result of Ore and Dubreil
\cite{cliffordpreston} says that a semigroup $S$ has a group of 
left quotients if
and only if it is right reversible and cancellative, where {\em right
reversible}
means that for any $a,b\in S$, $Sa\cap Sb\neq \emptyset$. 
The notion of group of left quotients was extended to that of 
 semigroup of 
left quotients
 by Fountain and Petrich in \cite{fountainpetrich:1986}; this
idea has been
extensively developed by a number of authors.
If $Q$ is a semigroup of
quotients of a subsemigroup $S$, then every $q\in G$ can be written 
as $q=a^{\sharp}b$ where $a,b\in S$ and $a^{\sharp}$
is the inverse of $a$ {\em in a subgroup of } $Q$. It is 
a hard problem  to
obtain a general characterisation of those semigroups possessing 
a semigroup of left quotients.  The best results in this direction
 consider semigroups having a semigroup of left quotients {\em 
lying in a particular class}; the result of Ore and Dubreil being such
an example.

The object of this paper is to begin a systematic review of an
alternative, but very natural,
notion of quotient. The focus here will
be on inverse semigroups, and we aim to develop a concept of quotient
that
will utilise the natural involution that an inverse 
semigroup possesses.
For an element $a$ of an inverse semigroup
$Q$, $a^{-1}$ will always denote the inverse of $a$ in the sense of inverse
semigroup theory. If $a$ lies in a subgroup, then
$a^{\sharp}=a^{-1}$, but $a^{-1}$ may exist without $a$ lying in a
subgroup.

\begin{defn}\label{defn:leftiorder} Let $S$ be a subsemigroup of an
inverse
semigroup $Q$. Then $S$ is a {\em left I-order} in $Q$ and
$Q$ is a {\em semigroup of left I-quotients} of $S$ if every $q\in Q$ can
be written as $q=a^{-1}b$ where $a,b\in S$. 
\end{defn}

We stress that this notion is not new - it has been 
effectively defined by a number of
authors, without being made fully explicit. Perhaps the first time this
idea appeared was in \cite{clifford:1953}, an article that (in our
terminology) considered right
cancellative monoids as left I-orders in bisimple inverse monoids. 
The results of \cite{clifford:1953} are arrived at via explicit construction
of quotients from equivalence classes of ordered pairs of elements of
$S$.
An alternative approach uses inverse hulls of right cancellative
monoids, and was pursued (and taken further)
in, for example, \cite{nivat:1970,mcalister:1976} and
\cite{petrich:1987}.  For inverse semigroups that do not have an
identity, one cannot take the right unit subsemigroup as a natural right cancellative submonoid. To overcome this, Reilly \cite{reilly:1968}, introduced the notion of quotient of an RP-system, where an RP-system corresponds (according to his result) to an
$\ar$-class of a bisimple inverse semigroup $Q$. Reilly's approach is subsumed in Lawson's
use of category actions to construct inverse semigroups \cite{lawson:1999}.
It is easy to 
see (and we show this enroute in Section~\ref{sec:preliminaries}) 
that if $Q$ is a bisimple inverse semigroup 
and $R$ is an $\mathcal{R}$-class of $Q$,
then every element of $Q$ 
can be written as $a^{-1}b$ where $a$ and $b$ lie in $R$; 
from the above, $R$ is an
RP-system but will not, in general, be a subsemigroup. Nevertheless,
Reilly successfully used RP-systems to characterise congruences
on bisimple inverse semigroups \cite{reilly:1971}.
In a different direction, Gantos \cite{gantos:1971} extended the work of
Clifford
to semilattices of right cancellative monoids.

The above mentioned articles, and others, all consider
left I-orders in particular classes of semigroups. Here,
after Section~\ref{sec:preliminaries} of preliminaries, 
we begin in Section~\ref{sec:iorders} in a rather more abstract way, by asking
the natural questions that arise when one introduces a notion of 
quotient. For example, if $S$ is a left
$I$-order in $Q$, under what conditions is $a^{-1}b=c^{-1}d$
where $a,b,c,d\in S$? 
If $S$ possesses a semigroup of left I-quotients, when is this unique?
 We then apply our findings in a
number
of different ways. We first show that Brandt semigroups of
left I-quotients of a given semigroup $S$ are unique up to isomorphism, thus extending
the result of \cite{fountainpetrich:1985} that states that
Brandt semigroups of left quotients of a given $S$ are unique. 

In Section~\ref{sec:inversehulls} we focus on left ample semigroups.
A semigroup $S$ is left ample if and only if it embeds into an inverse
semigroup $Q$ in such a way that if $a\in S$, then $aa^{-1}\in S$. 
Right cancellative monoids are
precisely left ample semigroups possessing a single idempotent.
A left ample semigroup $S$  has a natural representation as partial one-one
maps of $S$, from which we can construct its inverse hull $\Sigma(S)$. We
find necessary and sufficient conditions for 
a left ample semigroup to be a left I-order in its inverse hull, namely
that
for any $a,b\in S$, $Sa\cap Sb=Sc$ for some $c\in S$; we call
this Condition (LC). Thus, (LC) is a rather stronger condition 
than being right reversible. 
Our result
corresponds exactly to that of Clifford for right cancellative monoids.

We divert a little in Section~\ref{sec:semilattices} to look at strong
semilattices  of left ample
semigroups with (LC). Let $S$ be a strong semilattice $Y$ of
left ample semigroups $S_{\alpha}, \alpha\in Y$,
such that each $S_{\alpha}$ has (LC);
using the general results of Section~\ref{sec:iorders}, we show that 
$S$ is a left I-order in $Q$, where $Q$ a strong semilattice $Y$ of the inverse hulls
$\Sigma(S_{\alpha})$ of the semigroups $S_{\alpha},\alpha\in Y$,
if and only if the connecting morphisms are {\em LC-preserving}, and
this is equivalent to $S$ having (LC). In this case, $Q$ is the inverse hull of $S$.
Part of our result extends that of \cite{gantos:1971} in a rather simple
way.

In the final section, we consider left I-orders in bisimple inverse
semigroups. Building on the work of Section~\ref{sec:inversehulls},
we define a category $\mathbf{LAC}$ of left ample semigroups 
with (LC), and the category $\mathbf{BIS}$ of bisimple inverse semigroups, and
show that $\mathbf{LAC}$ and $\mathbf{BIS}$ are equivalent. This result may be 
be specialised to show that the corresponding category of right 
cancellative monoids is equivalent to the category of bisimple inverse
monoids.

It will be useful in future work to make adjustments to Definition~\ref{defn:leftiorder}
in the case where $Q$ has a zero. To avoid complications in the current
paper
we make no further mention of this. 

\section{Preliminaries and inverse hulls}\label{sec:preliminaries}

For any semigroup $Q$ we
denote the quasi-orders associated with Green's
relations $\ar$ and $\el$ by $\leqar$ and $\leqel$,
respectively. To avoid ambiguity
we may use the superscript $Q$ to
indicate that a relation applies to $Q$, so that, for example,
$a\leqarq b$ if and only if $aQ^1\subseteq bQ^1$.

The relation $\ars$ is defined on a semigroup $S$ by the rule that for any
$a,b\in S$, $a\,\ars\, b$ in $S$ if and only if $a\,\ar\, b$ in some
oversemigroup
of $S$. The following alternative characterisation of $\ars$ is
well known.

\begin{lem}\label{lem:ars} The following are equivalent for elements
$a,b$ of a semigroup $S$:

(i) $a\,\ars\, b$;

(ii) for all $x,y\in S^1$,
\[xa=ya
\mbox{ if and only if }xb=yb.\]\end{lem}

It is easy to see that $\ars$ is a left congruence,
$\ar\subseteq\ars$
and  $\ar=\ars$  if $S$ is regular. In general, however,
the 
inclusion can be  strict.

In a semigroup with commuting idempotents,  it is clear that any 
$\ars$-class contains at most one idempotent.
Where it exists we denote the (unique)
idempotent  in the $\ars$-class
of $a$ by $a^+$.
 If {\em every} $\ars$-class
contains an idempotent, $^+$ is then a unary operation on $S$ and
we may regard $S$ as an algebra of type $(2,1)$; as
such, morphisms must preserve the unary operation of $^+$ (and hence
the relation $\ars$). We may refer
to such morphisms as `$(2,1)$-morphisms' if there is danger
of ambiguity.  Of course, any
semigroup isomorphism must preserve $^+$.
We remark here that if $S$ is inverse, then
$a^+=aa^{-1}$ for all $a\in S$.

\begin{defn} A semigroup $S$ is
{\em left ample} if $E(S)$ is a semilattice, every $\ars$-class contains
a (necessarily unique) idempotent $a^+$ and
the {\em left ample} identity (AL) holds:
\[xy^+=(xy^+)^+x\,\,\,\,\,\,\,\mbox{(AL)}.\]
\end{defn}

It is easy to see that a semigroup is left ample and
unipotent (that is, it contains exactly one idempotent) if and only if
it a right cancellative monoid.

\begin{rem}\label{rem:qv1} The class of left 
ample semigroups forms a
quasi-variety of algebras of type $(2,1)$.
\end{rem}

Since any inverse semigroup is left ample, any subsemigroup of
such that is closed under $^+$ must therefore be left ample.
The converse is also true: left ample semigroups are (up to
isomorphism),
precisely
the submonoids of (symmetric) inverse semigroups closed under $^+$. As the
representation we use is needed in some later sections, we give a
brief outline. Further details can be found in, for example, 
 \cite{gould:notes}. It is useful to recall that in a symmetric inverse semigroup
 $\mathcal{I}_X$ we have 
 $\alpha\,\ar\, \beta\mbox{ if and only if }\dom\alpha=\dom\beta$,
 and
  $\alpha\,\el\, \beta\mbox{ if and only if }\im\alpha=\im\beta.$

Let $S$ be left ample. We construct an embedding of $S$ into
the symmetric inverse semigroup $\mathcal{I}_S$ as follows. For
each $a\in S$ we let $\rho_a\in \mathcal{I}_S$ be given by 
\[\dom \rho_a=Sa^+\mbox{ and }\im \rho_a=Sa\]
and for any $x\in \dom \rho_a$.
\[x\rho_a=xa.\]
Then the map $\theta_S:S\rightarrow \mathcal{I}_S$ is a $(2,1)$-embedding.

\begin{defn}\label{defn:inversehull} Let $S$ be a left ample semigroup.
Then the {\em inverse hull} $\Sigma(S)$ of $S$ is the inverse subsemigroup 
of $\mathcal{I}_S$ generated by $\im \theta_S$.
\end{defn}

We pause to consider a special case. Let $M$ be a right cancellative
monoid. Then for any $a\in M$, we have that $\rho_a:M\rightarrow Ma$, so
that $\dom \rho_a=M=\dom I_M$, giving that $\im \theta_S\subseteq R_1$, where
$R_1$ is the $\ar$-class of $I_M$ in $\mathcal{I}_M$.

\section{Left I-orders}\label{sec:iorders}

The classical notion of quotients in \cite{fountainpetrich:1986},
developed in a number of further articles, tells us that if we want 
close
relationship between a left order and its semigroup of left quotients,
then we may need to insist that the left order be {\em straight}
\cite{gould:2003}. Extrapolating this idea gives us the following:

\begin{defn}\label{defn:straight} Let $S$ be a left I-order in $Q$. Then
$S$ is {\em straight} in $Q$ if every $q\in Q$ can be written as
$q=a^{-1}b$ where $a,b\in S$ and $a\,\ar\, b$ in $Q$.
\end{defn}

If $S$ is a left I-order in $Q$ and $S$ has a right identity, then
this must be an identity of $Q$ (and hence of $S$).  
For any $q\in Q$ we have that
$q=a^{-1}b$ where $a,b\in S$, so that
\[qe=a^{-1}be=a^{-1}b=q\]
and
\[eq=ea^{-1}b=(ae)^{-1}b=a^{-1}b=q.\]

We remark that if a semigroup $S$ is a left order in $Q$ in the sense of
\cite{fountainpetrich:1986}, then it is certainly a left I-order. For, if
$S$ is a left order in $Q$, then we insist that any $q\in Q$ can
be written as $q=a^{\sharp}b$ where $a,b\in S$ and
$a^{\sharp}$ is the inverse of $a$ {\em in a subgroup of }
$S$.
Our notion of left I-order is more general, as we now
demonstrate with an easy example. 

\begin{ex}\label{ex:bicyclic} Let $B$ be the Bicyclic Semigroup and let $S=R_{(0,0)}$, the
$\mathcal{R}$-class of the identity. It is clear that $S$ is
a subsemigroup of $B$. For any $(a,b)\in B$ we have that
\[(a,b)=(a,0)(0,b)=(0,a)^{-1}(0,b)\]
so that $S$ is  a left I-order in $B$. On the other hand, the only
element of $S$ lying in a subgroup is $(0,0)$ and
$(0,0)^{\sharp}(0,n)=(0,n)$, for any $(0,n)\in S$. Thus
$S$ is not a left order in $B$.
\end{ex}

The fact that $R_{(0,0)}$ is a left I-order in $B$ is a very special 
case of the result of \cite{clifford:1953}
 mentioned in the Introduction, which we shall revisit. The semigroup
$B$ is bisimple and we shall see that bisimple inverse semigroups play
an important role in this theory. Suppose that $Q$ is bisimple and we
pick an $\ar$-class $R=R_e$ of $Q$, where $e\in E(S)$. Let $q\in Q$. As
$Q$ is bisimple we can find mutually inverse elements
$x,x^{-1}\in Q$   such that 
$xx^{-1}=e$ and $x^{-1}x=qq^{-1}$.
Then $q=x^{-1}xq$ and $xq\,\ar\, xqq^{-1}=x\,\ar\, e$. Thus
any element of $Q$ can be written as a quotient of elements
{\em chosen from any $\ar$-class}.

For an example of a different flavour, we present the following. For
later
purposes, it is useful to recall that if $B$ is a Brandt semigroup
and $a,e\in B$ with $e=e^2$, then $ea\neq 0\, (ae\neq 0)$ implies that
$ea=a\, (ae=a)$.

\begin{ex}\label{ex:brandt} Let $H$ be a left order in a group $G$, and let $\mathcal{B}^0=\mathcal{B}^0(G,I)$ be a Brandt semigroup over  $G$
where $|I|\geq 2$. Fix $i\in I$ and let \[S_i=\{ (i,h,j):h\in H, j\in I\}\cup \{ 0\}.\] Then $S_i$ is a straight left I-order in $\mathcal{B}^0$.

To see this, notice that $S_i$ is a subsemigroup, $0=0^{-1}0$, and for any
$(j,g,k)\in \mathcal{B}^0$, we may write $g=a^{-1}b$ where $a,b\in H$ and then
\[(j,g,k)=(i,a,j)^{-1}(i,b,k)\]
where $(i,a,j), (i,b,k)\in S_i$. 

Again, it is easy to see that $S_i$ is not a left order in $\mathcal{B}^0$.
\end{ex}

Notice that if $S$ is  a left I-order in $Q$ and
$a,b\in S$ with $a\,\ar\, b$ in $Q$, then $a^{-1}\,\ar\, a^{-1}b\,\el\, b$
in $Q$, so that
if $S$ is straight in $Q$, then {\em $S$ intersects every $\el$-class of
$Q$. }

In this initial article we will be primarily interested in left ample semigroups
that are
left I-orders. In such cases {\em the relation $\ars$ will always refer
to the left I-order}.

\begin{lem}\label{lem:amplestraight} Let $S$ be a left ample semigroup,
embedded (as a $(2,1)$-algebra) in an inverse semigroup $Q$. If
$S$ is a left I-order in $Q$, then $S$ is straight. 
\end{lem}
\begin{proof} Let $q=a^{-1}b\in Q$ where $a,b\in S$. Then
\[q=(a^+a)^{-1}(b^+b)=a^{-1}a^+b^+b=a^{-1}b^+a^+b=(b^+a)^{-1}(a^+b).\]
We have
\[a^+b\,\ars\, a^+b^+=b^+a^+\,\ars\, b^+a\]
and so $a^+b\,\ar\, b^+a$ in $Q$ and $S$ is straight.
\end{proof}

We know from the classical case that a semigroup may be a left I-order in non-isomorphic
semigroups of left I-quotients (see, for example, 
\cite{easdowngould:1996}).
For the remainder of this section we concentrate on determining when two
semigroups of straight left I-quotients of a given semigroup are
isomorphic. More generally we introduce the following notion.

\begin{defn}\label{defn:lifting} Let $S$ be a subsemigroup of $Q$ and let $\phi:S\rightarrow
P$ be a morphism from $S$ to a semigroup $P$. If there
is a morphism $\overline{\phi}:Q\rightarrow P$ such that
$\overline{\phi}|_S=\phi$, then we say that
$\phi$ {\em lifts} to  domain $Q$ and $\phi$ {\em lifts} to $\overline{\phi}$. If $\phi$ lifts to an isomorphism, then
we say that $Q$ and $P$ are {\em isomorphic over $S$}.
\end{defn}

To achieve our goal, we must first examine when two quotients
$a^{-1}b$ and $c^{-1}d$ are equal,
where $a,b,c,d\in S$ and $S$ is a left I-order in $Q$. Notice that if
$a\,\ar^Q\, b$ and $c\,\ar^Q\, d$ and $a^{-1}b=c^{-1}d$, then
an earlier remark gives that
$a\,\el^Q\, c$ and $b\,\el^Q\, d$. In 
Lemma~\ref{lem:equality} below, we give
conditions on $S$ such that $a^{-1}b=c^{-1}d$; the use of Green's
relations
in $Q$ in our conditions will be `internalised' to $S$ at a later point.
First, a preliminary remark that, given its usefulness, more than merits
the name of lemma.

\begin{lem}\label{lem:nassersuperlemma} Let $b,c,x,y$
be elements of an inverse semigroup $Q$ such that
$x\,\ar\, y$. If $bc^{-1}=x^{-1}y$, then $xb=yc$. 
\end{lem}
\begin{proof} We have that
\[bc^{-1}cb^{-1}=(bc^{-1})(bc^{-1})^{-1}=(x^{-1}y)(x^{-1}y)^{-1}=
x^{-1}yy^{-1}x=x^{-1}x\]
as $x\,\ar\, y$. Hence
\[bc^{-1}c=bb^{-1}bc^{-1}c=bc^{-1}cb^{-1}b=x^{-1}xb\]
and so $xbc^{-1}c=xb$. From $y=xbc^{-1}$ we have
\[xb=xbc^{-1}c=yc.\]
\end{proof}

\begin{lem}\label{lem:equality} Let $S$ be a straight left I-order in
$Q$. Let $a,b,c,d\in S$ with $a\,\ar^Q\, b$ and $c\,\ar^Q\, d$ in $Q$.
Then $a^{-1}b=c^{-1}d$ if and only if
there exist $x,y\in S$ with $xa=yc$ and
$xb=yd$ and such that $a\,\ar^Q\, x^{-1}$,
$x\,\ar^Q\, y$ and $y\,\el^Q\, c^{-1}$ in $Q$. 

\end{lem}
\begin{proof} Suppose first that $a^{-1}b=c^{-1}d$.
Then as above,
$b\,\el^Q\, d$  and $a\,\el^Q\, c$. Let $x,y\in S$ be such that $ac^{-1}=
x^{-1}y$ and $x\,\ar^Q\, y$. Then
\[b=ac^{-1}d=x^{-1}yd\]
so that $xb=yd$. From Lemma~\ref{lem:nassersuperlemma},
$xa=yc$. Also,
\[a\,\ar^Q\, aa^{-1}\,\ar^Q\, ac^{-1}=x^{-1}y\,\ar^Q\, x^{-1},\]
and 
\[y=xac^{-1}\,\el^Q\, a^{-1}ac^{-1}=c^{-1}cc^{-1}=c^{-1}\]
as required.

Conversely, if $xa=yc$, $xb=yd$ for some $x,y\in S$ with
$x\,\ar^Q\, y$, $a\,\ar^Q\, x^{-1}$ and $y\,\el^Q\, c^{-1}$, then
$a=x^{-1}yc, b=x^{-1}yd$ and 
\[\begin{array}{rcll}
a^{-1}b&=&(x^{-1}yc)^{-1}(x^{-1}yd)\\
&=&c^{-1}y^{-1}xx^{-1}yd\\
&=&c^{-1}y^{-1}yd&\mbox{ as }x\,\ar^Q\, y\\
&=&c^{-1}cc^{-1}d&\mbox{ as }y\,\el^Q\, c^{-1}\\
&=&c^{-1}d.\end{array}\]
\end{proof}

Let $S$ be a subsemigroup of an inverse semigroup
$Q$. We use Green's relations on $Q$ to 
define binary relations $\leqarqs,\ar^Q_S,
\leqelqs$ and $\el^Q_S$ and a ternary
relation
$\mathcal{T}^Q_S$ on $S$ by the rules that:
\[\leqarqs=\leq_{\mathcal{R}}^Q\cap (S\times S)\mbox{ and }\leqelqs
=\leq_{\mathcal{L}}^Q\cap 
(S\times
S),\]
so that $\leqarqs$ and $\leqelqs$ are, respectively, left and right
compatible quasi-orders. We then define $\ar^Q_S$ and
$\el^Q_S$ to be the associated equivalence relations, so that
\[\ar^Q_S=\ar^Q\cap (S\times S)\mbox{ and } \el^Q_S=\el^Q\cap (S\times S).\]
Consequently, $\ar^Q_S$ and $\el^Q_S$ are left and right compatible.
We define $\mathcal{T}^Q_S$ by the rule that for any $a,b,c\in S$,
\[(a,b,c)\in\mathcal{T}^Q_S\mbox{ if and only if }ab^{-1}Q\subseteq
c^{-1}Q.\]

\begin{lem}\label{lem:ttoleqar} Let $S$ and $T$ be subsemigroup
 of inverse
semigroups
$Q$ and $P$ respectively, and let $\phi:S\rightarrow T$
be a morphism. If
for all $a,b,c\in S$,
\[(a,b,c)\in \mathcal{T}^Q_S\Rightarrow (a\phi,b\phi,c\phi)\in 
\mathcal{T}^P_{T},\]
then for all $u,v\in S$,
\[u\leq_{\mathcal{R}}^Q v^{-1}\Rightarrow 
u\phi\leq_{\mathcal{R}}^P (v\phi)^{-1}.\]
\end{lem}
\begin{proof} Suppose that $u,v\in S$ and $uQ\subseteq v^{-1}Q$.
Then $uu^{-1}Q\subseteq v^{-1}Q$, so that $(u,u,v)\in \mathcal{T}^Q_S$.
By assumption, $(u\phi, u\phi,v\phi)\in \mathcal{T}^P_{T}$,
so that
\[u\phi P=u\phi(u\phi)^{-1}P\subseteq (v\phi)^{-1}P\]
and $u\phi \leq_{\mathcal{R}}^P (v\phi)^{-1}$ as required.
\end{proof}

We use the relation $\mathcal{T}^Q_S$ to prove our rather general result
below.
{\em As in the classical case, $\mathcal{T}^Q_S$ can be avoided in some special
cases of interest.}

\begin{thm}\label{thm:homs} Let $S$ be a straight left I-order in
$Q$ and let $T$ be a subsemigroup of an inverse
semigroup $P$. Suppose that $\phi:S\rightarrow T$
is a morphism. Then $\phi$
lifts
to a (unique) morphism $\overline{\phi}:Q\rightarrow P$ if and only if
for all $(a,b,c)\in S$:

(i) $(a,b)\in\ar^Q_S\Rightarrow (a\phi,b\phi)\in \ar^P_{T}$;

(ii) $(a,b,c)\in \mathcal{T}^Q_S\Rightarrow (a\phi,b\phi,c\phi)\in 
\mathcal{T}^P_{T}$.

If (i) and (ii) hold and $S\phi$ is a left I-order in $P$,
then $\overline{\phi}:Q\rightarrow P$ is onto. 
\end{thm}
\begin{proof} If $\phi$ lifts to a morphism $\overline{\phi}$, then 
as morphisms between inverse semigroups preserve inverses and Green's
relations, it is easy to see that $(i)$ and $(ii)$ hold.

Conversely, suppose that $(i)$ and $(ii)$ hold. We define
$\overline{\phi}:Q\rightarrow P$ by the rule that
\[(a^{-1}b)\overline{\phi}=(a\phi)^{-1}b\phi\]
where $a,b\in S$ and $a\,\ar^Q\, b$. 

To show that $\overline{\phi}$ is well defined, suppose that
\[a^{-1}b=c^{-1}d\]
where $a,b,c,d\in S$, $a\,\ar^Q\, b$ and $c\,\ar^Q\, d$. Then
by Lemma~\ref{lem:equality}, 
there exist $x,y\in S$ with $xa=yc$ and
$xb=yd$ and such that $a\,\ar^Q\, x^{-1}$,
$x\,\ar^Q\, y$ and $y\,\el^Q\, c^{-1}$. Applying
$\phi$, we have that $x\phi\, a\phi=y\phi\, c\phi$
and $x\phi\, b\phi=y\phi\, d\phi$. By $(i)$ we also
have that $x\phi\,\ar^P\, y\phi,
a\phi\,\ar^P\, b\phi$ and $c\phi\, \ar^P\, d\phi$,
and by $(ii)$ and
Lemma~\ref{lem:ttoleqar}, it follows that 
$b\phi \leq^P_{\mathcal{R}} (x\phi)^{-1}$
and $d\phi\leq^P_{\mathcal{R}}
(y\phi)^{-1}$. 

From $x\phi\, b\phi=y\phi\, d\phi$ we can now deduce
that $b\phi=(x\phi)^{-1}y\phi\, d\phi$ so that
\[\begin{array}{rcl}
(a\phi)^{-1}b\phi &=& (a\phi)^{-1}(x\phi)^{-1}
y\phi\, d\phi\\
&=& (x\phi\, a\phi)^{-1}y\phi\, d\phi\\
&=&(y\phi\, c\phi)^{-1} y\phi\, d\phi\\
&=& (c\phi)^{-1}(y\phi)^{-1}y\phi\, d\phi\\
&=&(c\phi)^{-1}d\phi,\end{array}\]
so that $\overline{\phi}$ is well defined.
 
To see that $\overline{\phi}$ lifts $\phi$, let $h\in S$;
then $h=k^{-1}\ell$ for some $k,\ell\in S$ with $k\,\ar^Q\, \ell$.
We have that  $kh=\ell$ and $h\leqarq k^{-1}$, so that
$k\phi h\phi=\ell\phi$ and by Lemma~\ref{lem:ttoleqar},
$h\phi\leq^P_{\mathcal{R}} (k \phi)^{-1}$. It follows that
$h\phi= (k\phi)^{-1}\ell \phi=h\overline{\phi}$.

We need to show that $\overline{\phi}$ is a morphism. To this end,
let $a^{-1}b,c^{-1}d\in Q$ with $a\,\ar^Q\, b$ and $c\,\ar^Q\, d$. 
By $(i)$ we have that $c\phi\,\ar^P\, d\phi$. Now
$bc^{-1}=u^{-1}v$ for some $u,v\in S$ with $u\,\ar^Q\, v$. By
Lemma~\ref{lem:nassersuperlemma}, $ub=vc$, so that
$u\phi\, b\phi=v\phi\, c\phi$. Further, $(b,c,u)\in \mathcal{T}^Q_S$,
so by assumption $(ii)$, we
have that $(b\phi,c\phi,u\phi)\in \mathcal{T}^P_{T}$.
Then from $u\phi\, b\phi (c\phi)^{-1}=v\phi\, c\phi (c\phi)^{-1}$
we obtain $b\phi (c\phi)^{-1}=(u\phi)^{-1} v\phi\, c\phi (c\phi)^{-1}$.

Multiplying, we have
\[(a^{-1}b)(c^{-1}d)=a^{-1}(bc^{-1})d=a^{-1}(u^{-1}v)d=(a^{-1}u^{-1})(vd)=(ua)^{-1}vd,\]
and
\[ua\,\ar^Q\, ub=vc\,\ar^Q\, vd.\]
Hence
\[\begin{array}{rcll}
((a^{-1}b)(c^{-1}d))\overline{\phi}&=&((ua)^{-1}vd)\overline{\phi}\\
&=&((ua)\phi)^{-1}(vd)\phi\\
&=&(a\phi)^{-1}(u\phi)^{-1}v\phi\, d\phi\\
&=&(a\phi)^{-1}(u\phi)^{-1}v\phi\, c\phi\, (c\phi)^{-1}d\phi\\
&=&(a\phi)^{-1}b\phi (c\phi)^{-1}d\phi\\
&=&(a^{-1}b)\overline{\phi}(c^{-1}d)\overline{\phi},
\end{array}\]
so that $\overline{\phi}$ is a morphism as required.

If $(i)$ and $(ii)$ hold and $S\phi$ is a left I-order in $P$,
then for any $p\in P$ we have $p=(a\phi)^{-1}b\phi$ for some
$a,b\in S$, so that $p=(a^{-1}b)\overline{\phi}$.
\end{proof} 

\begin{cor}\label{cor:iso} Let $S$ be a straight left I-order
in $Q$ and let $\phi:S\rightarrow P$ be an embedding of $S$
into an inverse semigroup $P$ such that $S\phi$ is a straight left
I-order in $P$. Then $Q$ is isomorphic to $P$ over $S$ if and only if
for any $a,b,c\in S$:

$(i)$ $(a,b)\in \ar^Q_S\Leftrightarrow (a\phi,b\phi)\in \ar^P_{S\phi}$;
  and

$(ii)$ $(a,b,c)\in \mathcal{T}^Q_S\Leftrightarrow (a\phi,b\phi,c\phi)\in
\mathcal{T}^P_{S\phi}$.

\end{cor}
\begin{proof} If $Q$ is isomorphic to $P$ over $S$
then $(i)$ and $(ii)$ hold from Theorem~\ref{thm:homs}.

Suppose now that $(i)$ and $(ii)$ hold. From
 Theorem~\ref{thm:homs}, $\phi$ lifts to a morphism
$\overline{\phi}:Q\rightarrow P$, where
$(a^{-1}b)\overline{\phi}=(a\phi)^{-1}b\phi$. Dually,
$\phi^{-1}:S\phi\rightarrow Q$ lifts to a morphism
$\overline{\phi^{-1}}:P\rightarrow Q$, where
$((a\phi)^{-1}b\phi)\overline{\phi^{-1}}=a^{-1}b$. Clearly
 $\overline{\phi}$ and $\overline{\phi^{-1}}$ are mutually inverse.
\end{proof}

Where $S$ is left ample, and $\phi$ preserves $^+$, then we note
that $(i)$ in Theorem~\ref{thm:homs} and Corollary~\ref{cor:iso}
is redundant. Further redundancies become apparent in the next section.

For an alternative use of Theorem~\ref{thm:homs}, we consider the 
case of left
I-orders in Brandt semigroups.

\begin{thm}\label{thm:brandhoms} Let $S$ be a left I-order in a Brandt
semigroup $\mathcal{B}^0=\mathcal{B}^0(G,I)$. Then $S$ 
contains a zero and is
straight in $\mathcal{B}^0$.

If $\phi:S\rightarrow T$ is an isomorphism where
$T$ is a left I-order in  $\mathcal{B}^0_1=\mathcal{B}^0(H,J)$, then
$\phi$ lifts to  an isomorphism $\overline{\phi}:\mathcal{B}^0\rightarrow
\mathcal{B}^0_1$.

\end{thm}
\begin{proof} Let $S$ be a left I-order in $\mathcal{B}^0$. Clearly
$S\neq \{ 0\}$. Suppose that
$S$ is contained within a non-zero group $\mathcal{H}$-class of 
$\mathcal{B}^0$, say $S\subseteq H_{(i,1,i)}$ (where $1$ is the identity
of
$G$). Then $S^{-1}S=\{ a^{-1}b:a,b\in S\}\subseteq H_{(i,1,i)}$,
a contradiction as we must have $0\in S^{-1}S$. It follows that either
$0\in S$, or there exists $(i,g,j)\in S$ for some $i,j\in I$ with $i\neq
j$ and $g\in G$. But in the latter case, we again have $
0=(i,g,j)(i,g,j)\in S$.

Clearly $0=0^{-1}0$. For any $(i,g,j)\in \mathcal{B}^0$, we have
$(i,g,j)=a^{-1}b$ for some $a,b\in S$. We must have that
\[a^{-1}=(i,u,\ell)\mbox{ and }b=(\ell, v,j)\]
so that $a=(\ell,u^{-1},i)\,\ar\, (\ell,v,j)=b$ in $\mathcal{B}^0$. Thus
$S$ is straight in $\mathcal{B}^0$.

Suppose now that $b\in S$ and $b\neq 0$. Then $bb^{-1}\neq 0$ and
$bb^{-1}=a^{-1}c$ for some $a,c\in S\setminus\{ 0\}$. Since
$aa^{-1}\neq 0$ and $\mathcal{B}^0$ is categorical at zero, we
have $ab\neq 0$. We deduce that $Sb\neq \{ 0\}$. 

Let $\phi:S\rightarrow T$ be as given. Let $a,b\in S\setminus\{ 0\}$
with
$a\,\ar\, b$ in $\mathcal{B}^0$. Then there exists $c\in S$ with
$ca\neq 0$ and hence $cb\neq 0$. It follows 
 that $(c\phi)(a\phi), (c\phi)(b\phi)$ are non-zero
in $\mathcal{B}^0_1$, so that $a\phi\,\ar\, b\phi$ in
$\mathcal{B}^0_1$. 

We also show that $\phi$ preserves $\mathcal{L}$; for if
$a,b\in S\setminus\{ 0\}$ and $a\,\el\, b$ in $\mathcal{B}^0$,
then $ba^{-1}\neq 0$, whence $ba^{-1}=u^{-1}v$ for some
$u,v\in S\setminus\{ 0\}$. It follows that $ub=va\neq 0$ and so
$(u\phi)(b\phi)=(v\phi)(a\phi)\neq 0$ in $T$. Consequently,
$a\phi\,\el\, b\phi$ in $\mathcal{B}^0_1$. 

It remains to show that $\phi$ preserves $\mathcal{T}^{\mathcal{B}^0}_S$. Suppose
therefore that
$a,b,c\in S$ and $ab^{-1}\mathcal{B}^0
\subseteq c^{-1}\mathcal{B}^0$. Then either $ab^{-1}=0$, 
or $ab^{-1}\,\ar\, c^{-1}$ in $ \mathcal{B}^0$. In the former case,
$a$ and $b$ are not $\el$-related in $\mathcal{B}^0$ and so by the
previous paragraph (applying the argument to $\phi^{-1}$), 
$a\phi$ and $b\phi$ are not $\el$-related in $\mathcal{B}^0_1$, giving
$(a\phi)(b\phi)^{-1}=0$ and so $(a\phi)(b\phi)^{-1}\mathcal{B}^0_1
\subseteq (c\phi)^{-1}\mathcal{B}^0$. On the
other hand, if $ab^{-1}\neq 0$, then we have
$a\,\el\, b$ and $a\,\ar\, c^{-1}$ in $\mathcal{B}^0$. It follows
that $ca\neq 0$ and so $a\phi\,\el\, b\phi$ and $(c\phi)(a\phi)
\neq 0$ in $\mathcal{B}^0_1$. Consequently,
\[(a\phi)(b\phi)^{-1}\mathcal{B}^0
=(a\phi)\mathcal{B}^0=(c\phi)^{-1}\mathcal{B}^0_1.\] 

Since $\phi$ (and, dually, $\phi^{-1}$) preserve both $\ar$ and $\mathcal{T}$, 
it follows from Corollary~\ref{cor:iso} that
$\phi$ lifts to an isomorphism $\overline{\phi}:
\mathcal{B}^0\rightarrow \mathcal{B}^0_1$.
\end{proof}

A characterisation of left I-orders in Brandt semigroups appears
as a consequence of  the study of left I-orders in primitive inverse semigroups
in
\cite{ghroda:primitive}. The statement of the result in the case
of Brandt
semigroups
was also communicated to the second author by A. Cegarra \cite{cegarra:private}.

\section{Inverse hulls of left ample semigroups}\label{sec:inversehulls}

Let $S$ be a left ample semigroup. Where convenient we identify $S$ with its image under
$\theta$ in $\Sigma(S)$. We begin with four simple but useful
observations.

\begin{rem}\label{rem:quotients} First observe that for any
$a,b\in S$, 
\[\begin{array} {rcl}
\dom \rho_a^{-1}\rho_b&=& (\im \rho_{a}^{-1}\cap \dom \rho_b)(\rho_a^{-1})^{-1}\\
&=&(\dom\rho_a\cap \dom\rho_b)\rho_a\\
&=&(Sa^+\cap Sb^+)\rho_a\\
&=&(Sa^+b^+)\rho_a\\
&=&Sb^+a\end{array}\]
and $\im \rho_a^{-1}\rho_b=(Sb^+a)\rho_a^{-1}\rho_b=Sb^+a^+\rho_b
=Sa^+b$, and for any $yb^+a\in Sb^+a$, 
\[(yb^+a)\rho_a^{-1}\rho_b=(yb^+a^+)\rho_b=ya^+b.\]
It follows that if $a\,\ars\, b$, then 
\[\dom\rho_a^{-1}\rho_b=Sa,\, \im \rho_a^{-1}\rho_b=Sb
\mbox{ and }
(ya) \rho_a^{-1}\rho_b= yb.\]
\end{rem}

\begin{rem}\label{rem:el}
We also observe that for any $a,b\in S$,
\[\begin{array}{rcl}
\rho_a\,\el\, \rho_b\mbox{ in }\Sigma(S)&\Leftrightarrow&
\im \rho_a=\im \rho_b\\
&\Leftrightarrow&Sa=Sb\\
&\Leftrightarrow&a\,\el\, b\mbox{ in }S.\end{array}\]
\end{rem}

\begin{rem}\label{rem:reverse} If $b,c\in S$ and $Sb\cap Sc=Sw$, where
$ub=vc=w$ and $ub^+=u, vc^+=v$, then
\[\dom \rho_b\rho_c^{-1}=(\im\rho_b\cap \dom\rho_c^{-1})\rho_b^{-1}=
(Sb\cap Sc)\rho_b^{-1}=Sw\rho_b^{-1}=(Sub)\rho^{-1}_b=Sub^+=Su,\]
and for any $su\in Su$, 
\[(su)\rho_b\rho_c^{-1}=(sub)\rho_c^{-1}=(svc)\rho_c^{-1}=
svc^+=sv,\]
so that in particular, $\im \rho_b\rho_c^{-1}=Sv$. Notice that
\[u=ub^+\,\ars\, ub=vc\,\ars\, vc^+=v.\]It follows
from Remark~\ref{rem:quotients} that $\rho_b\rho_c^{-1}=
\rho_u^{-1}\rho_v$.
\end{rem}

It is known \cite{clifford:1953}, although our terminology is new,
 that a right cancellative monoid is a left
I-order in its inverse hull if and only if it satisfies Condition (LC), which we
now define for arbitrary semigroups.

\begin{defn}\label{defn:lc} We say that a semigroup $S$
satisfies Condition (LC) if for any $a,b\in S$, there
exists $c\in S$ with $Sa\cap Sb=Sc$.
\end{defn}

Before proving the analogue of Clifford's result, we give two preliminary lemmas,  of which we will make much use.

\begin{lem}\label{lem:elcircars} For
any semigroup $S$,
$\ars\circ\el=\el\circ\ars$.
\end{lem}
\begin{proof} Let $a,b\in S$ with $a\,\ars\circ\el\, b$. Then
there exists an element $c\in S$ with $a\,\ars\,c\,\el\, b$. Either
$c=b$ in which case $a\,\el\, a\,\ars\, b$ or there exist
$u,v\in S$ with $c=ub$ and $b=vc$. Hence
$c=ub=uvc$ so that as $a\,\ars\, c$ we deduce
$a=uva$ and thus $a\,\el\, va$. But $va\,\ars\, vc=b$ so
that $a\, \el\circ\ars b$ and $\ars\circ\el\subseteq
\el\circ \ars$. The proof of the dual inclusion is very similar.
\end{proof}

\begin{lem}\label{lem:unionarclasses} 
Let $S$ be a left ample semigroup that is a left I-order
in an inverse semigroup $Q$, such that $S$ is a union of $\ar$-classes
of $Q$. Then

(i) $S$ is a $(2,1)$-subalgebra of $Q$;

(ii) for $a,b\in S$ with $a\,\ars\, b$, $a^{-1}b$ is idempotent if and
only if $a=b$;

(iii) for any $a,b\in S$, $Sa\subseteq Sb$ if and only if
$Qa\subseteq Qb$;

(iv) for any $a,b,c\in S$, $Sa\cap Sb=Sc$ if and only if
$Qa\cap Qb=Qc$;

(v) $S$ satisfies Condition (LC);

(vi) $Q$ is bisimple if and only if
\[\el^S\circ \ars=S\times S;\]
and

(vii) $Q$ is simple if and only if for all $a,b\in S$ there exists
$c\in S$ with 
\[a\,\ars\, c\leq^S_{\mathcal{L}} b.\]

\end{lem} 
\begin{proof} $(i)$ We need only show that if $a\in S$, then
$aa^{-1}=a^+$.
We have that $a\,\ar^Q\, aa^{-1}$ and $S$ is a union of
$\ar^Q$-classes, giving $aa^{-1}\in S$. As
$a\,\ars\, aa^{-1}$ we must have that $aa^{-1}=a^+$.

$(ii)$ If $a^{-1}b$ is idempotent, then
as $a^{-1}b\,\ar^Q\, a^{-1}a$, we must have
that $a^{-1}b=a^{-1}a$. Multiplying with $a$ on the
left gives $b=bb^{-1}b=aa^{-1}b=aa^{-1}a=a$. The converse is clear.

$(iii)$ If $a,b\in S$ and $Sa\subseteq Sb$, then clearly $Qa\subseteq
Qb$. On the other hand, if $Qa\subseteq Qb$, then we have that
$a=h^{-1}kb$ for some $h,k\in S$. It follows
that $a=((kb)^+h)^{-1}h^+kb$ and
\[((kb)^+h)^{-1}\,\ar^Q\, ((kb)^+h)^{-1}((kb)^+h)
\,\ar^Q\, ((kb)^+h)^{-1}h^+kb=a,\]
so that as $S$ is a union of $\ar^Q$-classes, 
$((kb)^+h)^{-1}\in S$. It follows that $Sa\subseteq Sb$.

$(iv)$ Suppose that $a,b\in S$ and $Sa\cap Sb=Sc$. Then
$c\in Sa\cap Sb\subseteq Qa\cap Qb$, so that $Qc\subseteq
Qa\cap Qb$. Conversely, if $h^{-1}ka=u^{-1}vb\in Qa\cap Qb$,
where $h,k,u,v\in S$, $h\,\ar\, k$ and $u\,\ar\, v$ in $Q$,
then $ka=hu^{-1}vb$ and
$hu^{-1}=s^{-1}t$ say, where $s,t\in S$ and $s\,\ar\, t$ in $Q$.
This gives that
\[ska=tvb\in Sa\cap Sb=Sc,\]
and so  $ska=tvb=xc$, where
$x\in S$. Now
\[ka=hu^{-1}vb=s^{-1}tvb=s^{-1}xc\]
and then $h^{-1}ka=h^{-1}s^{-1}xc\in Qc$. Hence
$Qa \cap Qb\subseteq Qc$ so that
$Qa\cap Qb=Qc$.

Conversely, suppose that $a,b\in S$ and $Qa\cap Qb=Qc$. From
$Qc\subseteq Qa$ and $Qc\subseteq Qb$, $(iii)$ gives that 
$Sc\subseteq Sa\cap Sb$. On the other hand, if $u=xa=yb\in Sa\cap Sb$
 for some
$x,y\in S$, then $u=qc$ for some $q\in Q$, whence
$Qu\subseteq Qc$. Again from $(iii)$, $Su\subseteq Sc$ so
that $Sa\cap Sb\subseteq Sc$ and we have
$Sa\cap Sb=Sc$ as required.

$(v)$  Let $a,b\in S$.
Then 
\[Qa\cap Qb=Qa^{-1}a\cap Qb^{-1}b=Qa^{-1}ab^{-1}b=Qab^{-1}b,\]
but $ab^{-1}=s^{-1}t$ for some $s,t\in S$ 
with $s\,\ar\, t$ in $Q$, and so
\[Qa\cap Qb=Qs^{-1}tb=Qtb.\]
From $(iv)$ we now have that
 $Sa\cap Sb=Stb$ and $S$ has Condition (LC).

$(vi)$ 
 We have observed in 
Lemma~\ref{lem:amplestraight} that $S$ is straight in $Q$. Let
$a^{-1}b,c^{-1}d\in Q$,  where
$a\,\ars\, b$ and $c\,\ars\, d$. Then
\[\begin{array}{rcl}
a^{-1}b\,\mathcal{D}^Q\, c^{-1}d
&\Leftrightarrow&a^{-1}b\,\ar^Q\, x^{-1}y\,\el^Q\, c^{-1}d
\mbox{ for some }x,y\in S\mbox{ with } x\,\ars\, y\\
&\Leftrightarrow& a^{-1}\,\ar^Q\, x^{-1}\mbox{ and }y\,\el^Q\, d
\mbox{ for some }x,y\in S\mbox{ with }x\,\ars\, y\\
&\Leftrightarrow& a\,\el^Q\, x\mbox{ and }y\,\el^Q\, d
\mbox{ for some }x,y\in S\mbox{ with }x\,\ars\, y\\
&\Leftrightarrow&a\,\el^S\, x\,\ars\, y\, \el^S\, d
\mbox{ for some }x,y\in S.
\end{array}\]
It follows that $Q$ is bisimple if and only if 
$\el^S\circ \ars\circ\el^S$ is universal. But from
Lemma~\ref{lem:elcircars}, $\el$ and $\ars$ commute on $S$, so
that $Q$ is bisimple if and only if $\el^S\circ \ars=S\times S$.

$(vii)$
Since $Q$ is inverse, it follows from 
\cite[Theorem 8.33]{cliffordpreston} that
$Q$ is simple if and only if for any $e,f\in E(Q)$, there is an element $q\in Q$
with $e=qq^{-1}$ and $q^{-1}q\leq f$. Let $e,f\in Q$
so that by $(ii)$, $e=a^{-1}a,f=b^{-1}b$
for some $a,b\in S$. Then $Q$ is simple if and only if there
exists $q=c^{-1}d$ (where $u,v\in S$ and $c\,\ars\, d$) such that
\[e=qq^{-1}=c^{-1}c\mbox{ and }d^{-1}d=q^{-1}q\leq f.\]
It follows that $Q$ is simple if and only if for any $a,b\in S$  there exist $c,d\in S$ with
$c\,\ars\, d$ such that $Sa=Sc$ and $Sd\subseteq Sb$. Again using
Lemma~\ref{lem:elcircars}, we obtain the given condition.
\end{proof}

We can now extend from right cancellative monoids to left ample semigroups the 
classic result for inverse hulls.

\begin{thm}\label{thm:leftampleorders} Let $S$ be a left ample
semigroup.
Then $S\theta_S$ is a left I-order in its inverse hull if and only if $S$ has
Condition (LC). 

If Condition (LC) holds, then $S\theta_S$ is a union of 
$R^{\Sigma(S)}$-classes.

\end{thm}
\begin{proof} Suppose that $S$ is a left I-order
in $\Sigma(S)$. The for any $b,c\in S$,
$\rho_b\rho_c^{-1}=\rho_u^{-1}\rho_v$ where
$u\,\ars\, v$. By Remark~\ref{rem:quotients}
$\dom (\rho_b\rho_c^{-1})=Su$, so that
\[Su=(\im \rho_b\cap \dom \rho_c^{-1})\rho_b^{-1}=
(Sb\cap Sc)\rho_b^{-1}.\]
But $\rho_b^{-1}\rho_b$ is the identity on $Sb=\im\rho_b$, and
so 
\[Sub=(Su)\rho_b=(Sb\cap Sc)\rho_b^{-1}\rho_b=Sb\cap Sc,\]
and $S$ has Condition (LC).

Conversely, suppose that $S$ has Condition (LC). Let
\[Q=\{ \rho_a^{-1}\rho_b:a,b\in S\}\subseteq \Sigma(S).\] Observe
that for any $a\in S$, $\rho_a=\rho_{a^+a}
=\rho_{a^+}^{-1}\rho_a$, so that $S\theta_S\subseteq Q$.

Consider $b,c\in S$. By Condition (LC), there exist
$u,v\in S$ with $Sb\cap Sc=Sub$ and $ub=vc$ with $ub^+=u$ and
$vc^+=v$. By Remark~\ref{rem:reverse}, $\rho_b\rho_c^{-1}=
\rho_u^{-1}\rho_v$. 

It follows that if $\rho_a^{-1}\rho_b,\rho_c^{-1}\rho_d\in 
Q$, then
\[(\rho_a^{-1}\rho_b)(\rho_c^{-1}\rho_d)=
\rho_a^{-1}(\rho_b\rho_c^{-1})\rho_d=
\rho_a^{-1}(\rho_u^{-1}\rho_v)\rho_d=
(\rho_u\rho_a)^{-1}(\rho_v\rho_d)=
\rho_{ua}^{-1}\rho_{vd},\]
so that $Q$ is closed under multiplication. Clearly $Q$ is closed under
taking inverses, so that $\Sigma(S)\subseteq Q$ from definition of inverse hull,
and 
so $Q=\Sigma(S)$ as required.

Finally, if $e\in E(S)$ and $\rho_e\,\ar^{\Sigma(S)}\, \rho_a^{-1}\rho_b$, where 
$a,b\in S$ and $a\,\ars\, b$, then
$\dom \rho_e=\dom \rho_a^{-1}\rho_b$, so that $Se=Sa$ 
and $a$ is regular in $S$. Any inverse $c$ of 
 $a$ in $S$ must be such that $\rho_c$ is the unique inverse of $\rho_a$
in $Q$, so that
$\rho_a^{-1}\in S\theta_S$ and hence $\rho_a^{-1}\rho_b\in S\theta_S$.
\end{proof}

\begin{cor}\label{cor:bisimple} The following conditions
are equivalent for a left ample semigroup $S$:

(i) $\Sigma(S)$ is bisimple;

(ii) $S$ has Condition (LC) and
$\ars\circ\el=S\times S$;

(iii) $S$ is a left I-order in $\Sigma(S)$ and $\ars\circ\el=S\times S$.

\end{cor}
\begin{proof} We recall that the embedding of
$S$ into $\Sigma(S)$ is via what, in the terminology of
\cite{mcalister:1976}, are called one-one {\em partial right translations}.
It follows that $\Sigma(S)$ is an inverse
subsemigroup of the inverse semigroup
$\hat{S}$ of one-one partial right translations.
Thus for any $\alpha\in \Sigma(S)$, $\dom \alpha$ is a left
ideal and for any $a\in \dom\alpha$ and $x\in S$, 
$(xa)\alpha=x(a\alpha)$.

$(ii)\Rightarrow (iii)$ and $(iii)\Rightarrow (i)$ are immediate from
Lemma~\ref{lem:unionarclasses} and
Theorem~\ref{thm:leftampleorders}.

 $(i)\Rightarrow (ii)$. Suppose that $\Sigma(S)$ is bisimple, and let 
$e\in E(S)$.
For
any $\alpha\in\Sigma(S)$, we know that $\alpha\,\dee\, \rho_e$,
so that $\alpha\,\ar\, \beta\,\el\, \rho_e$ in
$\Sigma(S)$. Then
$\dom\alpha=\dom\beta$ and $\im\beta=Se$. It follows that
$\dom\beta=S(e\beta^{-1})$ so that $\dom \alpha$ is principal. Now
let $a,b\in S$; then 
\[\dom (\rho_a\rho_b^{-1})=(Sa\cap Sb)\rho_a^{-1}=Sw\]
for some $w\in S$,
and so
\[Swa=Sa\cap Sb\]
and $S$ has (LC). From Theorem~\ref{thm:leftampleorders},
$\ars\circ\el$ is universal on $S$.
\end{proof}

We recall that a left ample semigroup $S$ is {\em proper} if
$\ars\,\cap\, \sigma=\iota$, where $\sigma$ is the least right
cancellative
congruence on $S$, and where $\sigma$ is given by the formula that for any $a,b\in S$,
\[a\,\sigma\, b\Leftrightarrow ea=eb\mbox{ for some }e\in E(S).\]
Clearly, if $S$ is a subsemigroup in an inverse semigroup $Q$,
then if $a\,\sigma\, b$ in $S$, we have that $a\,\sigma\, b$ in
$Q$, but the converse may not be true. In other words, 
there is a natural morphism from
$S/\sigma$ to $Q/\sigma$, but this may not be an embedding.

\begin{thm}\label{thm:eunitary} Let $S$ be
a left ample semigroup such that $S$ is a left I-order in $Q$
where $S$ is a union of $\ar$-classes of $Q$. Then
the following conditions are equivalent:

(i) 
 $Q$ is E-unitary;

(ii) $S$ is proper and  $S/\sigma$ embeds naturally in $Q/\sigma$;

(iii) $S$ is proper and $S/\sigma$ is cancellative.

\end{thm}
\begin{proof}

$(i)\Rightarrow (ii)$ Suppose that $Q$ is
E-unitary, and $a,b\in S$ are such that $a\,\sigma\, b$ {\em in }
$Q$. Then $ea=eb$ for some $e\in Q$, so that
$eb^+a=ea^+b$. But $b^+a\,\ar^Q\, a^+b$ and so $b^+a=a^+b$.
This gives that $a\,\sigma\, b$ {\em in } $S$. 

Clearly, if $a,b\in S$ and $a\, (\ars\cap\sigma)\, b$ in $S$, then
$a\, (\ar\cap\sigma)\, b$ in $Q$, whence $a=b$ and $S$ is proper.

$(ii)\Rightarrow (iii)$ This is clear.

$(iii)\Rightarrow (i)$ Let $a^{-1}b,c^{-1}d\in Q$, where $a,b,c,d\in S$,
$a\,\ars\, b$ and $c\,\ars\, d$. Suppose that $a^{-1}b
\, (\ar\cap\sigma)\, c^{-1}d$ in $Q$. Then there exists $x\in S$
such that 
\[x^{-1}xa^{-1}b=x^{-1}xc^{-1}d\]
and $a^{-1}b\,\ar^Q\, c^{-1}d$. From the former, 
$xa^{-1}b=xc^{-1}d$ and from the latter, $a\,\el\, c$ in $Q$. Hence
$a\,\el\, c$ in $S$ and so there exist $u,v\in S$ with
$a=uc$ and $c=va$. We may choose $u,v$ such that
$a^+u=u$ and $c^+v=v$. Now $a=uc=uva$ so that
$a^+=uva^+$, whence $u=a^+u=uva^+u=uvu$. Similarly, $v=vuv$,
so that $u$ and $v$ are mutually inverse in both $S$ and $Q$. 

From $xa^{-1}b=xc^{-1}d$ we have that
\[xa^{-1}b=x(va)^{-1}d=xa^{-1}v^{-1}d=xa^{-1}ud.\]
But $xa^{-1}\,\el^Q\, y$ for some $y\in S$, so that
$yb=yud$ and as $S/\sigma$ is cancellative, $b\,\sigma\, ud$ in $S$.
Also, $b\,\ars\, a=uc\,\ars\, ud$ so that as $S$ is proper,
$b=ud$. Now 
\[a^{-1}b=a^{-1}ud=a^{-1}v^{-1}d=(va)^{-1}d=c^{-1}d\]
and $Q$ is $E$-unitary as required.

\end{proof}

We remark that if the conditions of Theorem~\ref{thm:eunitary} hold,
then for any $q=[a^{-1}b]\in Q/\sigma$, we have that
$q=[a]^{-1}[b]$ and so the cancellative monoid $S/\sigma$ is a left
order in the group $Q/\sigma$.

The following result is classic;
most of it follows from 
Theorem~\ref{thm:leftampleorders} 
and Corollary~\ref{cor:bisimple}.

\begin{cor}\label{cor:rcleftorders}\cite{clifford:1953,nivat:1970,mcalister:1976}. The
following conditions are equivalent for a right cancellative monoid $S$:

(i) $\Sigma(S)$ is bisimple;

(ii) $S$ has Condition (LC);

(iii) $S$ is a left I-order in $\Sigma(S)$.

If the above conditions hold, then
$S$ is the $\ar$-class of the identity of $Q$. Further, $\Sigma(S)$ is
$E$-unitary if and only if $S$ is cancellative.

Conversely, the $\ar$-class of the identity of any bisimple inverse
monoid
is right cancellative with Condition (LC).
\end{cor}
\begin{proof} The equivalence of $(i),(ii)$ and $(iii)$ 
follows from Corollary~\ref{cor:bisimple} and the fact
that $\ars$ is universal on $S$.

Suppose that $(i),(ii)$ and $(iii)$ hold. Let $e$ be the identity of $S$. 
As remarked in Section~\ref{sec:preliminaries},
$\Sigma(S)$ is a monoid with identity $e$. Since $S$ is a single
$\ars$-class, and the embedding of $S$ into
$\Sigma(S)$ is a $(2,1)$-embedding, we have
$S\subseteq R^{\Sigma(S)}_e$. Again by
Theorem~\ref{thm:leftampleorders},
we have that $R^{\Sigma(S)}_e\subseteq S$, so that
$S=R^{\Sigma(S)}_e$. 

Since $\sigma=\iota$ on $S$, it is clear that $S$ is proper and
$S/\sigma\cong S$. From Theorem~\ref{thm:eunitary}
$\Sigma(S)$ is $E$-unitary if and only if $S$ is cancellative.

Conversely, let $R$ be the $\ar$-class of the identity of a bisimple
inverse monoid $Q$. It is easy to see that $R$ is a right cancellative
monoid, and from a comment in Section~\ref{sec:preliminaries},
we have that $R$ is a left I-order in $Q$.
Lemma~\ref{lem:unionarclasses}
tells us that $R$ has (LC).
\end{proof}

We now give a promised simplification of Theorem~\ref{thm:homs}.
First, we say that a $(2,1)$-morphism $\phi:S\rightarrow T$, where 
$S$ and $T$ are
 left ample
semigroups with Condition (LC) is {\em (LC)-preserving}
if, for any $b,c\in S$ with $Sb\cap Sc=Sw$, we have that
\[T(b\phi)\cap T(c\phi)=T(w\phi).\]
This condition is not new: it appeared originally in 
\cite{warne:1964} for right cancellative monoids with (LC), where it was called 
an {\em sl homomorphism} and subsequently (or variations thereof, and under different names) in, 
for example, \cite{gantos:1971} and \cite{mcalister:1976}. Using the
fact that for idempotents $e,f$ of an inverse semigroup $Q$, we have
that $Qe\cap Qf=Qef$, it is easy to verify that any morphism between
inverse semigroups is (LC)-preserving.

The following result was first proved in the special case of $S$ and $T$ being right cancellative in \cite{warne:1964}.

\begin{thm}\label{thm:hullhoms} Let $S$ and $T$ be left ample
semigroups with Condition (LC) and let $Q$ and $P$ be their inverse
hulls. Suppose that $\phi:S\rightarrow T$ is a $(2,1)$-morphism. Then
$\phi$ lifts to a morphism $\overline{\phi}:Q\rightarrow P$ if
and only if $\phi$ is (LC)-preserving.
\end{thm}
\begin{proof} For ease in this proof we identify $S$ and $T$ with
$S\theta_S$ and $T\theta_T$, respectively. We have remarked that any such $\phi$ preserves
$\ars$, and since $(\ars)^S=\ar^Q\cap(S\times S)$ and 
$(\ars)^T=\ar^P\cap (T\times T)$,  $(i)$ of Theorem~\ref{thm:homs} holds. It remains
to show that $(ii)$ of that theorem holds if and only if
$\phi$ is (LC)-preserving.

Suppose first that $\phi$ is (LC)-preserving. If $(a,b,c)\in \mathcal{T}^Q_S$,
then $ab^{-1}Q\subseteq c^{-1}Q$. Now $S$ has (LC) so that
$Sa\cap Sb=Sw$ for some $w\in S$ and $ua=vb=w$ for some
$u,v\in S$ with $ua^+=u$ and $vb^+=v$. From Remark~\ref{rem:reverse},
$ab^{-1}=u^{-1}v$ and $u\,\ars\, v$. Hence $u^{-1}vQ\subseteq c^{-1}Q$
so that $Su\subseteq Sc$ from Lemma~\ref{lem:unionarclasses}. Clearly, $T(u\phi)
\subseteq T(v\phi), u\phi a\phi =v\phi b\phi=w\phi,
u\phi (a\phi)^+=u\phi$ and $v\phi (b\phi)^+=v\phi$. As $\phi$ is (LC)-preserving,
$T(a\phi)\cap T(b\phi)=T(w\phi)$ whence $a\phi(b\phi)^{-1}
=(u\phi)^{-1}v\phi$ and it follows that $(a\phi,b\phi,c\phi)\in
\mathcal{T}^P_T$.

Conversely, suppose that $(ii)$ of Theorem~\ref{thm:homs} holds, so that
$\phi$ lifts to a morphism $\overline{\phi}:Q\rightarrow P$. Supppose
that $b,c\in S$ and $Sb\cap Sc=Sw$. We have $ub=vc=w$ for some
$u,v\in S$ with $ub^+=u,vc^+=v$ and $u\,\ars\, v$. This gives
that $bc^{-1}=u^{-1}v$ and so, applying $\overline{\phi}$,
$b\phi (c\phi)^{-1}=(u\phi)^{-1}v\phi$. As $T$ has (LC), we certainly
have that $b\phi (c\phi)^{-1}=h^{-1}k$ for some $h,k\in T$ with
$h(b\phi)=k (c\phi)=z, T(b\phi)\cap T(c\phi)=Tz$
and $h\,\ars\, k$ in $T$. From Lemma~\ref{lem:unionarclasses},
$u\phi\,\el\, h$ in $T$, so that $w\phi=(ub)\phi=u\phi\, b\phi\,\el\, h(b\phi)=z$
in $T$. We now have that
\[T(b\phi)\cap T(c\phi)=Tz=T(w\phi)\]
and $\phi$ is (LC)-preserving.
\end{proof}

The above result could (via a series of intermediate steps) be deduced 
from 
Theorem 2.6 of \cite{mcalister:1976}. For, the ample condition ensures 
that a left ample semigroup is embedded in the semigroup $\hat{S}$
 of one-to-one partial right translations of $S$ via the right regular 
representation described in Section~\ref{sec:preliminaries}. Further,
the image of $S$ is contained in $J(S)$, the set of join irreducible elements of
$\hat{S}$. By \cite[Proposition 1.14]{mcalister:1976}, if $S$ has (LC), then $J(S)$ is an inverse semigroup, which is isomorphic to our $\Sigma(S)$. The
restriction of $\theta$ in \cite[Theorem 2.6]{mcalister:1976} to $J(S)$,
with a slight adaptation of the notion of {\em permissible homomorphism}, will now
give our Theorem~\ref{thm:hullhoms}.

\section{Semilattices of inverse semigroups}\label{sec:semilattices}

We begin by setting up our notation. Let $Y$ be a semilattice and
let $S$ be a semigroup such that $S$ is the disjoint union of
subsemigroups $S_{\alpha},\alpha\in Y$, and is such that for any
$\alpha,\beta\in Y$, $S_{\alpha}S_{\beta}\subseteq S_{\alpha\beta}$.
Then we say that $S$ is a {\em semilattice $Y$ of subsemigroups
$S_{\alpha},\alpha\in Y$} and write
$S=\mathcal{S}\big(Y;S_{\alpha}\big)$.
{\em We make the convention that if we write $x_{\alpha}\in 
S$ (for any symbol $x$ and any $\alpha\in Y$), then
we mean that $x_{\alpha}\in S_{\alpha}$.}

If there exists a set of morphisms
$\phi_{\alpha,\beta}:S_{\alpha}\rightarrow
S_{\beta}$ for $\alpha\geq \beta$ such that

$(i)$ $\phi_{\alpha,\alpha}=I_{S_{\alpha}}$ for all $\alpha\in Y$;\\
and

$(ii)$ $\phi_{\alpha,\beta}\phi_{\beta,\gamma}=\phi_{\alpha, \gamma}$
for all
$\alpha,\beta,\gamma\in Y$ with $\alpha\geq\beta\geq \gamma$,
such that the binary operation in $S$ is given by the rule that
\[a_{\alpha}b_{\beta}=(a_{\alpha}\phi_{\alpha,\alpha\beta})(b_{\beta}\phi_{\beta,
\alpha\beta}),\]
where  the last
product is taken in $S_{\alpha\beta}$, then we say that
$S$ is a {\em strong semilattice $Y$ of semigroups
$S_{\alpha},\alpha\in Y$, with connecting morphisms
$\phi_{\alpha,\beta},
\alpha\geq \beta$} and write
$S=\mathcal{S}\big(Y;S_{\alpha};\phi_{\alpha,\beta}\big)$.

Let $Q=\big(Y;Q_{\alpha};\psi_{\alpha,\beta}\big)$
be a strong semilattice of {\em bisimple inverse monoids}
$Q_{\alpha},\alpha\in Y$, such that the connecting morphisms are
monoid morphisms. It follows that
the set $E$ of identities $E=\{ e_{\alpha}:\alpha\in Y\}$ forms
a subsemigroup, indeed a semilattice isomorphic
to $Y$.  Let $R_{\alpha}$ denote the $\ar$-class of
$e_{\alpha}$ in $Q_{\alpha}$. Then it is easy to see
that $S=\mathcal{S}\big(Y;R_{\alpha};\phi_{\alpha,\beta}\big)$ is a strong semilattice of right
cancellative monoids, where $\phi_{\alpha,\beta}=
\psi_{\alpha,\beta}|_{R_{\alpha}}$. In \cite{gantos:1971}, 
Gantos showed how to recover the structure of $Q$ from that of $S$; in
our
terminology, $Q$ is a semigroup of left I-quotients of $S$
and the morphisms $\phi_{\alpha,\beta}$ satisfy Condition
(LC).

In this section we revisit and generalise Gantos's result. We believe
that its correct context is that of inverse hulls of left ample
semigroups, and we show that his result can be naturally extended to
strong
semilattices
of left ample semigroups. Gantos uses an explicit construction of
quotients
involving ordered pairs subject to an equivalence relation - we avoid
all such technicalities by using our results concerning lifting of morphisms.

We first
observe that the `strong' in Gantos's result is automatic. The proof of
the following is entirely routine, but we provide it for completeness.

\begin{lem}\label{lem:strong} 
Let $P=\mathcal{S}\big(Y;M_{\alpha}\big)$ where
each $M_{\alpha}$ is a monoid with identity
$e_{\alpha}$, such that $E=\{ e_{\alpha}:\alpha\in Y\} $
is a subsemigroup of $P$. Then $E$ is a semilattice isomorphic to $Y$
and $E$ is central in $P$.

If we define $\phi_{\alpha,\beta}:M_{\alpha}\rightarrow M_{\beta}$ by
$a_{\alpha}\phi_{\alpha,\beta}=a_{\alpha}e_{\beta}$, 
where $\alpha\geq \beta$, then each $\phi_{\alpha,\beta}$ is a monoid
morphism, and $P=\mathcal{S}\big(Y;M_{\alpha};\phi_{\alpha,\beta}\big)$.
\end{lem}
\begin{proof} Let $a_{\alpha}\in M_{\alpha}$ and suppose first that $\alpha\geq
\beta$. Then
\[a_{\alpha}e_{\beta}=e_{\beta}(a_{\alpha}e_{\beta})=
(e_{\beta}a_{\alpha})e_{\beta}=e_{\beta}a_{\alpha}.\]
Now, for arbitrary $e_{\gamma}$,
\[a_{\alpha}e_{\gamma}=(a_{\alpha}e_{\gamma})e_{\alpha\gamma}=
a_{\alpha}(e_{\gamma}e_{\alpha\gamma})=
a_{\alpha}e_{\alpha\gamma}=e_{\alpha\gamma}a_{\alpha}=
(e_{\alpha\gamma}e_{\gamma})a_{\alpha}=
e_{\alpha\gamma}(e_{\gamma}a_{\alpha})=
e_{\gamma}a_{\alpha},\]
so that $E$ is central in $P$. 

It is easy to see that for $\alpha\geq \beta$, $\phi_{\alpha,\beta}
:M_{\alpha}\rightarrow M_{\beta}$ is a monoid morphism, 
$\phi_{\alpha,\alpha}=I_{M_{\alpha}}$ and for
$\alpha\geq\beta\geq \gamma$, 
$\phi_{\alpha,\gamma}=\phi_{\alpha,\beta}\phi_{\beta,\gamma}$.
Let $Q=\mathcal{S}\big(Y;M_{\alpha};
\phi_{\alpha,\beta}\big)$ and denote the binary operation
in $Q$ by $*$.

For $a_{\alpha},b_{\beta}\in M$ we
have
\[a_{\alpha}*b_{\beta}=(a_{\alpha}\phi_{\alpha,\alpha\beta})
(b_{\beta}\phi_{\beta,\alpha\beta})=
(a_{\alpha}e_{\alpha\beta})(b_{\beta}e_{\alpha\beta})=
(a_{\alpha}b_{\beta})e_{\alpha\beta}=
a_{\alpha}b_{\beta},\]
as required.
\end{proof}

\begin{prop}\label{prop:slample} Let $S=\mathcal{S}\big(
Y;S_{\alpha};\phi_{\alpha,\beta}\big)$,
where each $S_{\alpha}$ is left ample and the connecting morphisms are 
$(2,1)$-morphisms.

(i) The semigroup $S$ is left ample, and for any $a,b\in S$,
$a\,\ars\, b$ in $S$ if and only if $
a,b\in S_{\alpha}$ for some $\alpha\in Y$ and
$a\,\ars\, b$ in $S_{\alpha}$.

(ii) If each $S_{\alpha}$ has (LC), then $S$ has
(LC) if and only if every $\phi_{\alpha,\beta}, \alpha\geq
\beta$, is (LC)-preserving.
\end{prop}
\begin{proof} $(i)$ Let $f_{\alpha},g_{\beta}\in E(S)$; then
\[f_{\alpha}g_{\beta}=(f_{\alpha}\phi_{\alpha,\alpha\beta})(
g_{\b}\phi_{\beta,\alpha\beta})=(g_{\beta}\phi_{\beta,\alpha\beta})(
f_{\alpha}\phi_{\alpha,\alpha\beta})=g_{\beta}f_{\alpha},\]
using the fact that $E(S_{\alpha\beta})$ is a semilattice. Thus
$E(S)$ is a semilattice.

Suppose now that $a_{\alpha}\,(\ars)^{S}\, b_{\beta}$. Let
$f_{\alpha}$ be the idempotent in the 
$(\ars)^{S_{\alpha}}$-class of $a_{\alpha}$.
Then as $f_{\alpha}a_{\alpha}=a_{\alpha}$ we must also
have that $f_{\alpha}b_{\beta}=
b_{\beta}$ so that $\beta\leq \alpha$. With the dual we
obtain that $\alpha=\beta$; clearly, then
$a_{\alpha}\,(\ars)^{S_{\alpha}}
\,
b_{\alpha}$.

Conversely, suppose that $a_{\alpha}\,(\ars)^{S_{\alpha}}
\,
b_{\alpha}$ and $x_{\gamma}a_{\alpha}=y_{\delta}a_{\alpha}$. Then
$\gamma\alpha=\delta\alpha=\mu$, say, and
$(x_{\gamma}\phi_{\gamma,\mu})(a_{\alpha}\phi_{\alpha,
\mu})
=(y_{\delta}\phi_{\delta,\mu})(a_{\alpha}\phi_{\alpha,\mu})$.
But $\phi_{\alpha,\mu}$ is a $(2,1)$-morphism, and 
$a_{\alpha}\,(\ars)^{S_{\alpha}}
\,
b_{\alpha}$, so that $a_{\alpha}\phi_{\alpha,\mu}\,(\ars)^{S_{\mu}}
\,
b_{\alpha}\phi_{\alpha,\mu}$. We thus obtain that 
$(x_{\gamma}\phi_{\gamma,\mu})(b_{\alpha}\phi_{\alpha,
\mu})
=(y_{\delta}\phi_{\delta,\mu})(b_{\alpha}\phi_{\alpha,\mu})$
and hence $x_{\gamma}b_{\alpha}=y_{\delta}b_{\alpha}$. Making an easy
adjustment
for $x_{\gamma}=1$ yields that 
$a_{\alpha}\, (\ars)^S\, b_{\alpha}$.

 Notice that from the above, there is no ambiguity
in the use of the superscript $^+$. 
To see that $S$ is left ample, let $a_{\alpha}\in S$ and
$f_{\beta}\in E(S)$. Then
\[(a_{\alpha}f_{\beta})^+a_{\alpha}=
((a_{\alpha}\phi_{\alpha,\alpha\beta})(f_{\beta}\phi_{\beta,\alpha\beta}))^+
(a_{\alpha}\phi_{\alpha,\alpha\beta})=
(a_{\alpha}\phi_{\alpha,\alpha\beta})(f_{\beta}\phi_{\beta,\alpha\beta}),\]
using the fact that $S_{\alpha\beta}$ is left ample, so that
$(a_{\alpha}f_{\beta})^+a_{\alpha}=a_{\alpha}f_{\beta}$ as
required.

$(ii)$ Suppose that each $S_{\alpha}$ has (LC). 

Assume first
that each $\phi_{\alpha,\beta}$ is (LC)-preserving. Let
$a_{\alpha},b_{\beta}\in S$ and let $\gamma=\alpha\beta$. As
$S_{\gamma}$ has (LC) we know that 
\[S_{\gamma}a_{\alpha}\cap S_{\g}b_{\beta}
=S_{\gamma}(a\phi_{\alpha,\gamma})
\cap S_{\gamma}(b\phi_{\beta, \gamma})=S_{\gamma}c_{\gamma},\]
for some $c_{\gamma}$. We claim that $Sa_{\alpha}\cap Sb_{\beta}=
Sc_{\gamma}$.

Certainly $c_{\gamma}=x_{\gamma}a_{\alpha}=y_{\gamma}b_{\beta}$
for some $x_{\gamma},y_{\gamma}\in S_{\g}$, so that
$c_{\g}\in Sa_{\a}\cap Sb_{\b}$ and so
\[Sc_{\g}\subseteq Sa_{\a}\cap Sb_{\b}.\] 

On the other hand, let $d\in Sa_{\a}\cap Sb_{\b}$; then there
are elements $u_{\mu}, v_{\nu}\in S$ with
$d=u_{\mu}a_{\a}=v_{\nu}b_{\b}$. Let $\tau=\mu\a=\nu\b$,
so that $\tau\leq \gamma$. Then
\[d=d_{\tau}=d_{\tau}^+d_{\tau}=
(d^+_{\tau}u_{\mu})a_{\a}=(d^+_{\tau}v_{\nu})b_{\b}\in S_{\tau}a_{\a}
\cap S_{\tau}b_{\b}.\] 
Now $\phi_{\g,\tau}$ is (LC)-preserving, so that
$S_{\tau}a_{\a}\cap S_{\tau}b_{\b}=S_{\tau}c_{\g}$. This
gives that $d=z_{\tau}c_{\g}\in Sc_{\g}$. Hence 
$Sa_{\alpha}\cap Sb_{\beta}=
Sc_{\gamma}$ as required.

\bigskip

Conversely, assume that $S$ has (LC) and suppose
that $\alpha\geq \beta$; we must show that
$\phi_{\a,\b}$ is (LC)-preserving. 

 We first show that for any $a_{\a},b_{\a},c_{\a}\in S$, 
 \[S_{\a}a_{\a}\cap S_{\a}b_{\a}=S_{\a}c_{\a}\Leftrightarrow
Sa_{\a}\cap Sb_{\a}=Sc_{\a}.\]

\noindent($\Leftarrow$) If $Sa_{\a}\cap Sb_{\a}=Sc_{\a}$, we have that
\[c_{\a}=ua_{\a}=vb_{\a}=(ua^+_{\a})a_{\a}=(vb^+_{\a})b_{\a}
\in S_{\a}a_{\a}\cap S_{\a}b_{\a},\]
so that $S_{\a}c_{\a}\subseteq S_{\a}a_{\a}\cap S_{\a}b_{\a}$.
On the other hand, if $x_{\a},y_{\a}\in S_{\a}$ and
\[x_{\a}a_{\a}=y_{\a}b_{\a}\in S_{\a}a_{\a}\cap S_{\a}b_{\a}
\subseteq Sa_{\a}\cap Sb_{\a},\]
then \[x_{\a}a_{\a}=y_{\a}b_{\a}=zc_{\a}
=(zc_{\a}^+)c_{\a}\in S_{\a}c_{\a}.\]
Thus
$S_{\a}a_{\a}\cap S_{\a}b_{\a}\subseteq S_{\a}c_{\a}$
and we have $S_{\a}a_{\a}\cap S_{\a}b_{\a}= S_{\a}c_{\a}$
as desired.

\noindent($\Rightarrow$) Conversely, suppose that $S_{\a}a_{\a}
\cap S_{\a}b_{\a}=S_{\a}c_{\a}$.
We also know that $Sa_{\a}\cap Sb_{\a}=Sd_{\beta}$ for some $d_{\b}\in S$.
As $d_{\b}\in Sa_{\a}$ we have that $\beta\leq \a$, but
$c_{\a}\in Sd_{\b}$, so that $\a=\b$. From ($\Leftarrow$)
we have that $S_{\a}a_{\a}
\cap S_{\a}b_{\a}=S_{\a}d_{\a}$, so that $c_{\a}\,\el\, d_{\a}$ in
$S_{\a}$ and hence in $S$. Consequently,
$Sa_{\a}\cap Sb_{\a}=Sc_{\a}$.

We now return to the argument that $\phi_{\a,\b}$ is (LC)-preserving (for
$\alpha\geq\b$). Suppose that 
 $a_{\a},b_{\a},c_{\a}\in S$ and 
 $S_{\a}a_{\a}\cap S_{\a}b_{\a}=S_{\a}c_{\a}$. We know that
$c_{\a}=x_{\a}a_{\a}=y_{\a}b_{\a}$ for some $x_{\a},y_{\a}$, so that
$c_{\a}\phi_{\a,\b}=(x_{\a}\phi_{\a,\b})(a_{\a}\phi_{\a,\b})=
(y_{\a}\phi_{\a,\b})(b_{\a}\phi_{\a,\b})$, giving that
\[c_{\a}\phi_{\a,\b}\in S_{\b}(a_{\a}\phi_{\a,\b})\cap 
S_{\b}(b_{\a}\phi_{\a,\b})=S_{\b}d_{\b}\] for some
$d_{\b}$. From the above, $Sa_{\a}\cap Sb_{\a}
=Sc_{\a}$. We have that for some $u_{\b},v_{\b}$,
\[d_{\b}=u_{\b}(a_{\a}\phi_{\a,\b})=
v_{\b}(b_{\a}\phi_{\a,\b})= u_{\b}a_{\a}=v_{\b}b_{\a}\in Sa_{\a}\cap
Sb_{\a}=Sc_{\a}\]
so that $d_{\b}=z_{\gamma}c_{\a}=((z_{\gamma}c_{\a})^+z_{\gamma})
(c_{\a}\phi_{\a,\b})\in S_{\b}(c_{\a}\phi_{\a,\b})$. 
It follows that
\[S_{\b}(a_{\a}\phi_{\a,\b})\cap 
S_{\b}(b_{\a}\phi_{\a,\b})=S_{\b}d_{\b}=S_{\b}(c_{\a}\phi_{\a,\b})\]
and $\phi_{\a,\b}$ has (LC).
\end{proof}

We can now give the main result of this section.

\begin{thm}\label{thm:semilatticesleftample} Let
$S=\mathcal{S}\big(Y;S_{\alpha};\phi_{\alpha,\beta}\big)$ be a 
strong semilattice of left ample semigroups $S_{\alpha}$,
such that  the connecting morphisms
are $(2,1)$-morphisms. Suppose that each
$S_{\alpha},\alpha\in Y$ has (LC) and that $S$ has (LC).

For each $\alpha\in Y$, let $Q_{\alpha}$ be the inverse hull of
$S_{\alpha}$. Then for any $\alpha,\beta\in Y$ with $\alpha\geq \beta$,
we have that $\phi_{\alpha,\beta}$ lifts to a morphism
$\overline{\phi_{\alpha,\beta}}:Q_{\alpha}
\rightarrow Q_{\beta}$. Further, 
$Q=\mathcal{S}\big(Y;Q_{\alpha};\overline{\phi_{\alpha,\beta}}\big)$
is a strong semilattice of inverse semigroups, such
that $S$ is a straight left I-order in $Q$. 

Moreover,
$Q$ is isomorphic to the inverse hull of $S$. 
\end{thm}
\begin{proof} By Theorem~\ref{thm:leftampleorders},
each $S_{\alpha}\theta_{S_{\alpha}}$ is a left I-order
in its inverse hull - we identify $S_{\alpha}$ with
$S_{\alpha}\theta_{S_{\alpha}}$ and write the
inverse hull of $S_{\alpha}$ as $Q_{\a}$. By
Lemma~\ref{lem:amplestraight},
$S_{\a}$ is straight in $Q_{\a}$.

From Proposition~\ref{prop:slample}, $S$ is left ample
and as $S$ has (LC), the connecting
morphisms
are (LC)-preserving. By Theorem~\ref{thm:hullhoms}, each
$\phi_{\a,\b} (\a\geq\b)$ lifts to a morphism $\overline{\phi_{\a,\b}}:
Q_{\alpha}\rightarrow Q_{\b}$. Clearly $\overline{\phi_{\a,\a}}$ is the 
identity map and for any $\a\geq \b\geq \gamma$,
$\overline{\phi_{\a,\b}}\, \overline{\phi_{\b,\gamma}}=
\overline{\phi_{\a,\gamma}}$. Thus
$Q=\mathcal{S}\big(Y;Q_{\alpha};\overline{\phi_{\alpha,\beta}}\big)$
is a strong semilattice of inverse semigroups and $S$ is a
straight left I-order in $Q$.

It remains to show that $Q$ is isomorphic to the inverse hull
$P=\Sigma(S)$ of $S$. First, it is easy to check that $S$ is a union
of $\ar$-classes of $Q$.

For any $a,b\in S$,
\[\begin{array}{rcl}
a\,\ar^Q_S\, b&\Leftrightarrow& a,b\in S_{\a}\mbox{ for some }\alpha
\mbox{ and }a\,\ar^{Q_{\alpha}}_{S_{\alpha}}\, b\\
&\Leftrightarrow& a,b\in S_{\a}\mbox{ for some }\alpha
\mbox{ and }a\,(\ars)^{S_{\alpha}}\, b\\
&\Leftrightarrow& a\,(\ars)^S\, b\\
&\Leftrightarrow& a\theta_S\,\ar^P\, b\theta_S.\end{array}\]

Let $a_{\a},b_{\b}\in S$; we show that
\[Sa_{\a}\cap Sb_{\b}=S(a_{\a}\phi_{\a,\a\b})\cap
S(b_{\b}\phi_{\b,\a\b}).\]
Let \[x=u_{\gamma}a_{\alpha}=v_{\delta}b_{\beta}\in Sa_{\a}\cap Sb_{\b};\]
then
$\gamma\a=\delta \b=\tau$ say, so that $\tau\leq \a\b$ and
\[x=x^+x=(x^+u_{\gamma})a_{\a}=(x^+v_{\delta})b_{\b}=
(x^+u_{\gamma})(a_{\a}\phi_{\a,\tau})=
(x^+v_{\delta})(b_{\b}\phi_{\b,\tau})\]\[
=(x^+u_{\gamma})(a_{\a}\phi_{\a,\a\b})=
(x^+v_{\delta})(b_{\b}\phi_{\b,\a\b})\in S(a_{\a}\phi_{\a,\a\b})\cap
S(b_{\b}\phi_{\b,\a\b}).\]

Conversely, if
\[y=h_{\gamma}(a_{\a}\phi_{\a,\a\b})=k_{\delta}(b_{\b}\phi_{\b,\a\b})\in
S(a_{\a}\phi_{\a,\a\b})\cap S(b_{\b}\phi_{\b,\a\b})\]
then $\gamma\alpha\beta=\delta\a\b=\kappa$ say and
\[y=(y^+h_{\gamma})(a_{\a}\phi_{\a,\a\b})=
(y^+k_{\delta})(b_{\b}\phi_{\b,\a\b})=
(y^+h_{\gamma})a_{\a}=(y^+k_{\delta})b_{\b}\in Sa_{\a}\cap Sb_{\b}.\]

Now let $a,b,c\in S$. Consider $ba^{-1}\in Q$; say $b=b_{\b}$ and
$a=a_{\a}$. Then
\[ba^{-1}=(b\phi_{\b,\a\b})(a\phi_{\a,\a\b})^{-1}=x^{-1}y\]
where $x,y\in S_{\a\b}$, $x=x(b\phi_{\b,\a\b})^+,
y=y(a\phi_{\a,\a\b})^+$,
$S_{\a\b}(b\phi_{\b,\a\b})\cap S_{\a\b}(a\phi_{\a,\a\b})
=S_{\a\b}(x(b\phi_{\b,\a\b}))$ and
\[x(b\phi_{\b,\a\b})=y(a\phi_{\a,\a\b})=xb=ya.\]
Also, $x=x(b^+\phi_{\b,\a\b})=xb^+$ and similarly, $y=ya^+$.
From the proof of Proposition~\ref{prop:slample} and the argument above, we have that
$Sb\cap Sa=S(xb)$. It follows from Remark~\ref{rem:reverse}, that in $P$,
$b\theta_S(a\theta_S)^{-1}=(x\theta_S)^{-1}y\theta_S$.

Now
\[\begin{array}{rcl}
(a,b,c)\in \mathcal{T}^Q_S&\Leftrightarrow&ab^{-1}Q\subseteq c^{-1}Q\\
&\Leftrightarrow& Qba^{-1}\subseteq Qc\\
&\Leftrightarrow&Qx^{-1}y\subseteq Qc\\
&\Leftrightarrow&Qy\subseteq Qc\\
&\Leftrightarrow&Sy\subseteq Sc\\
&\Leftrightarrow& S\theta_S(y\theta_S)\subseteq S\theta_S(c\theta_S)\\
&\Leftrightarrow&P(y\theta_S)\subseteq P(c\theta_S)\\
&\Leftrightarrow& Pb\theta_S(a\theta_S)^{-1}\subseteq P(c\theta_S)\\
&\Leftrightarrow&(a\theta_S)(b\theta_S)^{-1}P\subseteq
(c\theta_S)^{-1}P\\
&\Leftrightarrow& (a\theta_S,b\theta_S,c\theta_S)\in P.\end{array}\]
From Corollary~\ref{cor:iso}, $Q$ is isomorphic to $P$ via an
isomorphism lifting $\theta_S$.
\end{proof}

From Lemma~\ref{lem:unionarclasses} and Theorem~\ref{thm:semilatticesleftample}
we have the following result of Gantos.

\begin{cor}\label{cor:gantos}
(cf. \cite[Main Theorem]{gantos:1971}) Let $S=\mathcal{S}\big(Y;S_{\alpha}\big)$ 
be a semilattice
$Y$ of right cancellative monoids $S_{\alpha}$ with identity
$e_{\alpha}$, such that 
each $S_{\alpha}$ has (LC). Suppose in addition that for any $\alpha\geq
\beta$, if $S_{\alpha}a_{\alpha}\cap S_{\alpha}b_{\alpha}=S_{\alpha}
c_{\alpha}$, then $S_{\beta}a_{\alpha}\cap S_{\beta}b_{\alpha}=S_{\beta}
c_{\alpha}$. For each $\alpha\in Y$, let $Q_{\alpha}$ be the inverse
hull of $S_{\alpha}$, so that $Q_{\alpha}$ is a bisimple inverse monoid,
and $S_{\alpha}$ is the $\ar^{Q_{\alpha}}$-class of $e_{\alpha}$.
Then $Q=\mathcal{S}\big(Y;Q_{\alpha}\big)$ is a 
semigroup of left I-quotients of $S$, such that
$E=\{ e_{\alpha}:\alpha\in Y\}$ is a subsemigroup.

Conversely, let $Q=\mathcal{S}\big(Y;Q_{\alpha}\big)$ be a semilattice
$Y$ of bisimple inverse monoids $Q_{\alpha}$,
with identity $e_{\alpha}$, such that $E=\{ e_{\alpha}:\alpha\in Y\}$ 
is a subsemigroup. Then
$S=\mathcal{S}\big(Y;R_{e_{\alpha}}\big)$ is a semilattice of
right cancellative monoids $R_{e_{\alpha}}$, such that each
$R_{e_{\alpha}}$ has (LC) and for any $\alpha\geq \beta$, if
$R_{e_{\alpha}}a_{\alpha}\cap R_{e_{\alpha}}b_{\alpha}
=R_{e_{\alpha}}c_{\alpha}$, then
$R_{e_{\beta}}a_{\alpha}\cap R_{e_{\beta}}b_{\alpha}
=R_{e_{\beta}}c_{\alpha}$.
\end{cor}

\section{Bisimple inverse semigroups}\label{sec:bisimple}

Let $Q$ be a bisimple inverse semigroup and let $S$ be a subsemigroup
of $Q$ that is a union of $\ar^Q$-classes. Clearly $S$ is left
ample and is embedded as a $(2,1)$-subalgebra of $Q$ and
from a remark
in Section~\ref{sec:preliminaries}, $S$ is a left I-order in $Q$. 
It follows
from Lemma~\ref{lem:unionarclasses} that $S$ has Condition (LC).

Let $\mathbf{BIS}$ be the category with objects ordered pairs
$(Q,S)$, where $Q$ is a bisimple inverse semigroup
and $S$ is a subsemigroup of $Q$ that is a union of
$\ar^Q$-classes. A morphism in $\mathbf{BIS}$ from
$(Q,S)$ to $(P,T)$ is a semigroup morphism
$\psi$ such that $S\psi\subseteq T$. Now let
$\mathbf{LAC}$ be the category with objects left ample semigroups 
with Condition (LC) and such that $\ars\circ\el$ is universal. A morphism
in $\mathbf{LAC}$ from
$S$ to $T$ is a $(2,1)$-morphism
$\phi:S\rightarrow T$ that is
(LC)-preserving. It is easy to see that $\mathbf{BIS}$ and
$\mathbf{LAC}$ are categories.

The next lemma follows from Theorem~\ref{thm:leftampleorders}.

\begin{lem}\label{lem:hulls} Let $S$ be an object in $\mathbf{LAC}$.
Then $(\Sigma(S),S\theta_S)$ is an object in $\mathbf{BIS}$.
\end{lem}

\begin{lem}\label{lem:fullhulls} (i) Let $(Q,S), (P,T)$ be  objects in
$\mathbf{BIS}$, and suppose that $\phi:S\rightarrow T$ is
an isomorphism. Then $\phi$ lifts to an isomorphism from $Q$ to $P$.

(ii) Let $(Q,S)$ be  an object in
$\mathbf{BIS}$. Then $\psi:\Sigma(S)\rightarrow Q$ given
by $(\rho_a^{-1}\rho_b)\psi=a^{-1}b$ is an isomorphism.
\end{lem}
\begin{proof} We need only prove $(i)$, for then $(ii)$ follows from
Lemma~\ref{lem:hulls}.

For ease we identify $S$ with $T$ and take $\phi$ to be the identity 
map on $S$. Notice that for any $a,b\in S$, $a\,\ar^Q\, b$ if and only if
$a\,\ars\, b$ in $S$, if and only if $a\,\ar^P\, b$. 
If $a\,\el^Q\, b^{-1}$, then $a\,\el^Q\, bb^{-1}=b^+$, so
from Lemma~\ref{lem:unionarclasses}, $a\,\el^S\, b^+$,
whence $a\,\el^P\, b^+\,\el^P\, b^{-1}$. Consequently,
$a\,\ar^Q\, b^{-1}$ if and only if $a^{-1}\,\el^Q\, b$
if and only if $a^{-1}\,\el^P\, b$ if and only if
$a\,\ar^P\, b^{-1}$.

We recall from Lemma~\ref{lem:equality} 
that for $a,b,c,d\in S$ with $a\,\ar^Q\, b$ and $c\,\ar^Q\, d$,
 $a^{-1}b=c^{-1}d$ in $Q$ if and only if
there exist $x,y\in S$ with $xa=yc$ and
$xb=yd$ and such that $a\,\ar^Q\, x^{-1}$,
$x\,\ar^Q\, y$ and $y\,\el^Q\, c^{-1}$. It follows from the
above observations that the rule that
takes $a^{-1}b\in Q$ (where $a,b\in S$ and $a\,\ars\, b$) to $a^{-1}b$ in $P$
is a bijection.

Suppose now that $bc^{-1}=x^{-1}y$ in $Q$, where
$b,c,x,y\in S$ and $x\,\ars\, y$. Notice that
$yc^+=y$. From Lemma~\ref{lem:nassersuperlemma},
$xb=yc$. Certainly $Sxb\subseteq Sb\cap Sc$. On the other hand,
if $ub=vc\in Sb\cap Sc$, then
\[ub=vc=ubc^{-1}c=ux^{-1}yc\]
so that $Qub\subseteq Qyc$ and so $Sub\subseteq Syc$. It follows
that $Sb\cap Sc=Sxb$. 

Conversely, if we are given that $Sb\cap Sc=Sxb$, where
$xb=yc$, $x\,\ars\, y$ and $yc^+=y$, then 
$b^{-1}bc^{-1}c=(xb)^{-1}xb=b^{-1}x^{-1}xb$
and so \[bc^{-1}=bb^{-1}x^{-1}xbc^{-1}=x^{-1}xbc^{-1}=
x^{-1}ycc^{-1}=x^{-1}y.\]

Suppose now that $a^{-1}b,c^{-1}d\in Q$ with
$a,b,c,d\in S$ and $a\,\ars\, b,c\,\ars\, d$. Then
$bc^{-1}=x^{-1}y$ in $Q$ with $x\,\ars\, y=yc^+$, $xb=yc$ and
$Sxb=Sb\cap Sc$. It follows that $bc^{-1}=x^{-1}y$ in $P$ also. Now
\[(a^{-1}b)(c^{-1}d)=a^{-1}(bc^{-1})d=a^{-1}(x^{-1}y)d=
(xa)^{-1}yd\]
in both $Q$ and $P$. It follows that 
the map that takes $a^{-1}b\in Q$ (where $a,b\in S$ and $a\,\ars\, b$) to 
$a^{-1}b\in P$ is an isomorphism, which clearly restricts to the
identity
on $S$.
\end{proof}

Let $F:\mathbf{LAC}\rightarrow \mathbf{BIS}$ be the functor that
takes an object $S$ of $\mathbf{LAC}$ to $(\Sigma(S),S\theta_S)$.  If $S,T$ are objects in $\mathbf{LAC}$ and
$\phi:S\rightarrow T$ is a morphism in $\mathbf{LAC}$, 
then by Theorem~\ref{thm:hullhoms}, $\phi$ lifts to a morphism
$\overline{\phi}:\Sigma(S)\rightarrow\Sigma(T)$.
More accurately, $\phi':S\theta_S\rightarrow
T\theta_T$ given by $\rho_a\phi'=\rho_{a\phi}$ lifts
to $\overline{\phi}$, so that  $(\rho_a^{-1}\rho_b)
\overline{\phi}=\rho_{a\phi}^{-1}\rho_{b\phi}$. Clearly
$S\theta_S\overline{\phi}=S\theta_S\phi'\subseteq T\theta_T$, so that
$\overline{\phi}$ is a morphism from $(\Sigma(S),S\theta_S)$
to $(\Sigma(T),T\theta_T)$ in $\mathbf{BIS}$. It is straightforward to
verify that $F$ is indeed a functor.

Let $G:\mathbf{BIS}\rightarrow \mathbf{LAC}$ 
take an object $(Q,S)$ of $\mathbf{BIS}$ to $S$ and a morphism
$\psi$ from $(Q,S)$ to $(P,T)$ in $\mathbf{BIS}$ to
$\phi=\psi|_S$. Clearly $\phi$ is a $(2,1)$-morphism.

\begin{lem}\label{lem:lc} The map $G$ defined as above is a functor
from $\mathbf{BIS}$ to $\mathbf{LAC}$.
\end{lem}
\begin{proof} We need only check that if $\psi$ is a morphism from
$(Q,S)$ to $(P,T)$ in $\mathbf{BIS}$, then $\phi=\psi_S$ is 
(LC)-preserving.

Let $a,b,c\in S$ be such that $Sa\cap Sb=Sc$. By 
Lemma~\ref{lem:unionarclasses}, $Qa\cap Qb=Qc$, so that as
$\psi$ is certainly an (LC)-morphism, $P(a\psi)\cap P(b\psi)=
P(c\psi)$ and again by Lemma~\ref{lem:unionarclasses},
and using the fact $\phi=\psi|_S$, 
$T(a\phi)\cap T(b\phi)=T(c\phi)$.

\end{proof}
 We now show that $FG$ and $GF$ are naturally isomorphic to 
$I_{\mathbf{LAC}}$ and $I_{\mathbf{BIS}}$, respectively. 

Let $S$ be any object in $\mathbf{LAC}$; then
\[SFG=(\Sigma(S),S\theta_S)G=S\theta_S\]
and $\theta_S:SI_{\mathbf{LAC}}\rightarrow SFG$ is an isomorphism
in $\mathbf{LAC}$. If $\phi:S\rightarrow T$ is a morphism in
$\mathbf{LAC}$, then
\[\phi FG=\overline{\phi}G=\overline{\phi}|_{S\theta_S}=\phi',\]
where $\rho_a\phi'=\rho_{a\phi}$. For any $s\in S$ we have that
\[s\theta_S\phi'=\rho_s\phi'=\rho_{s\phi}=s\phi\theta_T\]
so that

\begin{center}

\begin{pspicture}(-4,0)(12,5)
\psset{linewidth=1.5pt}
\psset{nodesep=1pt}

\rput(1,1){\rnode{T}{$T$}}
\rput(1,4){\rnode{S}{$S$}}
\rput(7,1){\rnode{T'}{$TFG$}}
\rput(7,4){\rnode{S'}{$S FG$}}
\rput(2,2.5){\rnode{U}{$$}}

\rput(6.5,2.5){\rnode{U'}{$$}}
\ncline{->}{S}{T}\Aput{$\phi$}
\ncline{->}{S'}{T'}\Aput{$\phi FG$}
\ncline{->}{S}{S'}\Aput{${\theta_S}$}
\ncline{->}{T}{T'}\Aput{${\theta_T}$}
\end{pspicture}\end{center}
commutes. 

On the other hand, for any object $(Q,S)$ of
$\mathbf{BIS}$, 
\[(Q,S)GF=SF=(\Sigma(S), S\theta_S)\]
and for a morphism $\psi:(Q,S)\rightarrow (P,T)$ in
$\mathbf{BIS}$,
\[\psi GF=\psi|_SF=\overline{\psi|_S},\]
where for $\rho_a^{-1}\rho_b\in \Sigma(S)$, 
$(\rho_a^{-1}\rho_b)\overline{\psi|_S}=\rho^{-1}_{a\psi}\rho_{b\psi}$.
By Lemma~\ref{lem:fullhulls}, $\mu_{(Q,S)}:(Q,S)
\rightarrow (\Sigma(S),S\theta_S)$ given by
$(a^{-1}b)\mu_{(Q,S)}=\rho_a^{-1}\rho_b$ is an isomorphims, which
clearly lies in $\mathbf{BIS}$. We have that for any 
$a^{-1}b\in Q$,
\[(a^{-1}b)\mu_{(Q,S)}(\psi GF)=
(\rho_a^{-1}\rho_b)(\psi GF)
=\rho_{a\psi}^{-1}\rho_{b\psi}=
\big((a\psi)^{-1}b\psi\big)\mu_{(P,T)}=
(a^{-1}b)\psi \mu_{(P,T)},\]
so that

\begin{center}\begin{pspicture}(-4,0)(12,5)
\psset{linewidth=1.5pt}
\psset{nodesep=1pt}

\rput(1,1){\rnode{T}{$(P,T)$}}
\rput(1,4){\rnode{S}{$(Q,S)$}}
\rput(7,1){\rnode{T'}{$(P,T) GF$}}
\rput(7,4){\rnode{S'}{$(Q,S)GF$}}
\rput(2,2.5){\rnode{U}{$$}}

\rput(6.5,2.5){\rnode{U'}{$$}}
\ncline{->}{S}{T}\Aput{$\psi$}
\ncline{->}{S'}{T'}\Aput{$\psi FG$}
\ncline{->}{S}{S'}\Aput{${\mu_{(Q,S)}}$}
\ncline{->}{T}{T'}\Aput{${\mu_{(P,T)}}$}
\end{pspicture}
\end{center}
commutes. 

We now have the main result of this section.

\begin{thm}\label{thm:catequiv} The categories $\mathbf{BIS}$
and $\mathbf{LAC}$ are equivalent.

\end{thm}

As a corollary, we have the classical result due to  Clifford and Warne, 
made explicit in \cite[Chapter X]{petrich:1984}. To state this
result, we let $\mathbf{BIM}$ be the full subcategory of
$\mathbf{BIS}$ consisting of all pairs $(Q,R_1)$, where
$Q$ is a bisimple inverse monoid and 
$R_1$ is the $\mathcal{R}$-class of the identity of $Q$,
and we let $\mathbf{RCC}$ be the full subcategory of $\mathbf{LAC}$
consisting of right cancellative monoids with the (LC) condition.
Since $F|_{\mathbf{RCC}}:\mathbf{RCC}\rightarrow \mathbf{BIM}$
and 
$G|_{\mathbf{BIM}}:\mathbf{BIM}\rightarrow \mathbf{RCC}$,
we deduce the following.

\begin{cor}\cite{clifford:1953,warne:1964,petrich:1984}\label{cor:rightcancellative} The categories
$\mathbf{BIM}$ and $\mathbf{RCC}$ are equivalent.
\end{cor}

Finally, we remark on the connection between this material and
Reilly's
RP-systems
\cite{reilly:1968}. Reilly defined an RP-system $(R,P)$ to be a 
`right partial semigroup' $R$ together with a subsemigroup $P$
of $R$ satisfying certain conditions. He showed that for
any RP-system $(R,P)$ there exists a bisimple inverse semigroup 
$Q(R,P)$,
some $\mathcal{R}$-class of which is isomorphic (under the appropriate
notion) to $R$. Conversely, for any $\mathcal{R}$-class $R$ of a bisimple
inverse
semigroup, there is a subsemigroup $P$ of $R$ such that $(R,P)$ is an 
RP-system and $Q(R,P)$ is isomorphic to $Q$. We remark that if $S$ is an
object in $\mathbf{LAC}$, then any $\ars$-class of $S$ is
an $\mathcal{R}$-class of a bisimple inverse semigroup and so
an RP-system. Conversely, if $(R,P)$ is an RP-system, let
$Q$ be a bisimple inverse semigroup into which it embeds as an 
$\mathcal{R}$-class $R_q$. Let $S$ be the smallest subsemigroup of
$Q$ containing $R_q$ that is a union of $\mathcal{R}$-classes of $Q$. 
Then $S$ is left ample, and a left I-order in $Q$, so that
$S$ is an object in $\mathbf{LAC}$.

\end{document}